\documentclass[11pt]{amsart}
\usepackage{mathabx}    
\usepackage{mathrsfs}  
\usepackage[backref,colorlinks]{hyperref} 
\usepackage{amsmath,amssymb,amsthm,bm}
\usepackage[abbrev,msc-links]{amsrefs}   
\usepackage[utf8]{inputenc}
\usepackage{tikz}
\usepackage{amsmath}
\usepackage{enumerate}
\usepackage{enumitem}

\usepackage{mathtools}

\DeclarePairedDelimiter\floor{\lfloor}{\rfloor}
\DeclarePairedDelimiter{\card}{|}{|}
\usepackage{lineno}
\theoremstyle{plain}
\newtheorem{thm}{Theorem}[section]

\newtheorem{prop}[thm]{Proposition}
\newtheorem{clm}[thm]{Claim}
\newtheorem{cor}[thm]{Corollary}
\newtheorem{lem}[thm]{Lemma}

\newtheorem*{clm*}{Claim}

\theoremstyle{definition}

\newtheorem{dfn}[thm]{Definition}
\newtheorem{exmp}[thm]{Example}

\newtheorem{obs}[thm]{Observation}
\usepackage{mathrsfs}

\numberwithin{equation}{section}

\normalsize                              
\def\COMMENT#1{}
\let\COMMENT=\footnote

\usepackage{amsmath}

\usepackage{enumitem}

\let\polishlcross=\l
\def\l{\ifmmode\ell\else\polishlcross\fi}

\let\eps=\varepsilon
\let\theta=\vartheta
\let\phi=\varphi

\newcommand{\mc}[1]{\mathcal{#1}}

\def\cH{{\mathcal H}}

\def\cF{{\mathcal F}}
\def\cB{{\mathcal B}}

\def\cG{{\mathcal G}}

\def\cK{{\mathcal K}}

\def\cW{{\mathcal W}}

\def\aut{\mathrm{aut}}

\DeclareMathAlphabet\bfc{OMS}{cmsy}{b}{n}
\DeclareMathAlphabet{\pzc}{OT1}{pzc}{m}{it}

\title{Vertex Ramsey properties of randomly perturbed graphs}
\author{Shagnik Das}
\author{Patrick Morris}
\author{Andrew Treglown}

\thanks{
SD: Freie Universit\"at Berlin, Germany, {\tt shagnik@mi.fu-berlin.de}.  Research supported by GIF grant \\ \indent G-1347-304.6/2016 and by the Deutsche Forschungsgemeinschaft (DFG) project 415310276. \\
\indent
PM: Freie Universit\"at Berlin, Germany and Berlin Mathematical School, Germany, {\tt pm0041@mi.fu-berlin.de}. \\ \indent Research supported by a Leverhulme Trust Study Abroad
  Studentship (SAS-2017-052$\backslash$9).\\
\indent AT: University of Birmingham, United Kingdom, {\tt a.c.treglown@bham.ac.uk}.}
\date{\today}

\begin{document}

\maketitle
\begin{abstract}
Given graphs $F,H$ and $G$, we say that 
$G$ is $(F,H)_v$-Ramsey if every
red/blue vertex colouring of $G$ contains
a red copy of $F$ or a blue copy of $H$.
Results of \L uczak, Ruci\'nski and Voigt,
and Kreuter determine the threshold
for the property that the random graph 
$G(n,p)$ is $(F,H)_v$-Ramsey.
In this paper we consider the sister problem in the setting of \emph{randomly perturbed graphs}.
In particular, we determine how many random
edges one needs to add to a dense graph to
ensure that with high probability the resulting
graph is $(F,H)_v$-Ramsey for all pairs $(F,H)$
that involve at least one clique.
\end{abstract}

\section{Introduction}
For $\ell \in \mathbb{N}$, a sequence of (not necessarily distinct) graphs $H_1,\ldots, H_\ell$, 
and a graph $G$, we say that $G$ is \emph{$(H_1,\dots,
H_ \ell)_v$-Ramsey} if for every $\ell$-colouring of the \emph{vertices} of $G$, there is some $i \in [\ell]$ for which
$G$ contains a copy of $H_i$ whose vertices are all coloured in the $i$th colour.
Similarly, 
we say $G$ is \emph{$(H_1,\dots, H_\ell)$-Ramsey} if for every $\ell$-colouring of the \emph{edges} of $G$, 
there is some $i \in [\ell]$ for which
$G$ contains a copy of $H_i$ whose edges are all coloured in the $i$th colour.
In the case when $\ell=2$ we take the convention that the colours used are red and blue.
If $H_1=\dots=H_\ell=H$ we write e.g.\ $(H_1,\dots,
H_\ell)_v$-Ramsey as $(H,\ell)_v$-Ramsey.

The classical question in Ramsey theory is to establish the smallest $n \in \mathbb N$
such that  the complete 
graph $K_n$ on $n$ vertices is $(H_1,\dots,H_\ell)$-Ramsey. In general, by Ramsey's theorem such an $n$ is known to always exist, but relatively few such \emph{Ramsey numbers} are known precisely.
In contrast to this, the analogous question  in the setting of vertex colourings is completely trivial.
Indeed, the pigeonhole principle implies $K_n$ is $(H_1,\dots,H_\ell)_v$-Ramsey if $n=v(H_1)+\dots +v(H_\ell)-\ell+1$, while it is not $(H_1,\dots,H_\ell)_v$-Ramsey
if $n$ is any smaller.

\subsection{Vertex Ramsey properties of random graphs}
On the other hand, vertex Ramsey questions are natural in the setting where the host graph is sparser. 
In particular, in the early 1990s \L uczak, Ruci\'nski and Voigt~\cite{luczak1992ramsey} investigated 
vertex Ramsey properties of random graphs. Recall that the random graph $G(n, p)$ has vertex set
$[n] := \{1, \dots , n\}$ and each edge is present with probability $p$, independently of all other choices.
To state their result we must introduce the notion of the $1$-density of a graph.
\begin{dfn}[1-density] \label{def:onedensity}
For a graph $H$, the \emph{$1$-density} of $H$ is defined to be 
\[m_1(H):=\max\left\{\frac{e_J}{v_J-1}: J\subseteq H, v_J\geq 2 \right\},\]
where here $e_J$ and $v_J$ denotes the number of edges and vertices in $J$ respectively.
\end{dfn}
The following result shows that the $1$-density of $H$ is the parameter that governs the threshold for 
the property that $G(n,p)$ is $(H,\ell)_v$-Ramsey. Recall that an event occurs in $G(n,p)$ with high probability (w.h.p.) if its probability tends to $1$ as $n\rightarrow \infty$. 
\begin{thm}[\L uczak, Ruci\'nski and Voigt~\cite{luczak1992ramsey}] \label{thm:onegraphthreshold}
Let $\ell \geq 2$ and let $H$ be a graph with at least one edge and that is not a matching if $\ell =2$.
Then there exist constants $c,C>0$, such that if $p\geq C n^{-1/m_1(H)}$, then w.h.p.\
$G(n,p)$ is  $(H,\ell)_v$-Ramsey and if $p\leq c n^{-1/m_1(H)}$, then w.h.p.\ $G(n,p)$ is not $(H,\ell)_v$-Ramsey.
\end{thm}
Intuitively, Theorem~\ref{thm:onegraphthreshold} states that a typical `sparse' $n$-vertex graph, i.e.\ one with density at most $cn^{-1/m_1(H)}$, is not $(H,\ell)_v$-Ramsey; whilst a typical `dense' $n$-vertex graph, i.e. one with density at least $Cn^{-1/m_1(H)}$, is $(H,\ell)_v$-Ramsey.

The following result of Kreuter~\cite{kreuter1996threshold} provides an asymmetric generalisation 
of Theorem~\ref{thm:onegraphthreshold}. To state it we require an asymmetric version of the $1$-density.
\begin{dfn}[Kreuter densities]\label{kreuterdensities}
Given two graphs $F$ and $H$ with at least one edge, such that $m_1(F)\leq m_1(H)$ we define \[m_K(F,H):=\max\left\{\frac{m_1(F)+e_J}{v_J}:J\subseteq H, v_J\geq 2\right\}.\]
\end{dfn}

Note that simple calculations imply that $m_1(F)\leq m_K(F,H)\leq m_1(H)$ and thus if $m_1(F)=m_1(H)$ then $m_K(F,H)$ is also the same value.

\begin{thm}[Kreuter~\cite{kreuter1996threshold}] \label{thm:kreuterthresholds}
Let  $\ell\geq 2$ and $H_1,\ldots, H_\ell$ be  graphs such that $m_1(H_1)\leq \ldots \leq m_1(H_\ell)$. Suppose further that  $H_{\ell-1}$ and $H_\ell$ contain at least one edge and $H_\ell$ is not a matching if $\ell=2$. Then there exist constants $c, C>0$ such that 
 \begin{itemize}
     \item if $p\geq Cn^{-1/m_K(H_{\ell-1},H_\ell)}$, then w.h.p.\ $G(n,p)$ is  $(H_1,\dots, H_\ell)_v$-Ramsey;
     \item if $p\leq cn^{-1/m_K(H_{\ell-1},H_\ell)}$, then w.h.p.\ $G(n,p)$ is not $(H_1,\dots, H_\ell)_v$-Ramsey.
\end{itemize}
\end{thm}
Note that there has been significant interest in (edge) Ramsey properties of random graphs also.
See~\cite{random3} for an analogue of Theorem~\ref{thm:onegraphthreshold} in this setting and 
e.g.~\cites{kohayakawa1997threshold, mns} for results
on asymmetric Ramsey properties of random graphs.

\subsection{Randomly perturbed graphs}
Theorems~\ref{thm:onegraphthreshold} and~\ref{thm:kreuterthresholds} give us information on the 
vertex Ramsey properties of \emph{typical graphs} of a given density. In this paper, we are interested in 
measuring \emph{how far away} a dense graph is from having a given vertex Ramsey property.
 The
model of randomly perturbed graphs, introduced by Bohman, Frieze and Martin~\cite{bfm1}, provides a
framework for studying such questions. In their model one starts with a dense graph and then adds
a given number of random edges to it. A natural problem in this setting is to determine how many
random edges are required to ensure that the resulting graph w.h.p.\
satisfies a given property.

Over recent years there has been a wealth of research in the area of randomly perturbed graphs, including
results on embedding (spanning) subgraphs 
(e.g.\ Hamilton cycles, spanning trees and subgraphs of bounded degree) in such 
graphs (see e.g.~\cites{ bwt2,  bfkm, bfm1,  bhkmpp, bmpp2, dudek, hmt, joos2, kks1, kst, nt}). 
In 2006, Krivelevich,
Sudakov and Tetali~\cite{kst} initiated the study of (edge) Ramsey properties of randomly perturbed graphs.
Combining their work with recent results of the first and third authors~\cite{dastreg} and of Powierski~\cite{power},
we now know the number of random edges one needs to add to an $n$-vertex graph of 
positive density to w.h.p.\ ensure the 
resulting graph is $(K_r,K_s)$-Ramsey for all values of $(r,s)$, except for the case when $r=4$ and $s \geq 5$.
See~\cite{dastreg} for other results on this topic.

In this paper, we   focus on vertex Ramsey properties of randomly perturbed 
graphs; in particular we resolve the $( K_r,H)_v$-Ramsey problem for $r\geq 2$ and arbitrary $H$. 
To state 
our results we first introduce the following notation.
\begin{dfn} \label{def:perturbedthresholdfunctions}
Fix some $d\in [0,1]$. Then for a pair of graphs $(F,H)$, we say that $p=p(n)$ is a 
\emph{perturbed vertex Ramsey threshold function} for the pair $(F,H)$ at density $d$ if:
\begin{enumerate}[label=(\roman*)]
    \item For any 
    $q(n)=\omega( p(n))$ and any  sequence $(G_n)_{n\in \mathbb{N}}$ of graphs of  density\footnote{Here we refer to the standard density of $G$. That is $d(G)=\frac{e_G}{\binom{v_G}{2}}$.} at least $d$ with $v_{G_n}=n$ for each $n\in 
    \mathbb N$, with high probability $G_n\cup G(n,q)$ is $(F,H)_v$-Ramsey.
    \item There exists a sequence of $n$-vertex graphs $(G_n)_{n\in \mathbb{N}}$ of density at least $d$, such that if 
    $q(n)=o( p(n))$, then with high probability $G_n\cup G(n,q)$ is not $(F,H)_v$-Ramsey.
    \end{enumerate}
    We denote by $p(n;F,H,d)$, the\footnote{As is the case in random graph theory, the threshold function is not uniquely determined but rather determined up to constants.} perturbed vertex Ramsey threshold function for $(F,H)$ at density $d$. If there exist $C,c>0$ such that $q(n)\geq Cp(n)$ suffices for (i) and $q(n)\leq cp(n)$ suffices for (ii), we say that the threshold function is \emph{sharp}.
    If it is the case that every sufficiently large graph of density at least $d$
    is $(F,H)_v$-Ramsey then we define $p(n;F,H,d):=0$.
\end{dfn}
Note we can analogously define the perturbed vertex Ramsey threshold function
for the $\ell$-coloured case; that is,
 given graphs $H_1, \dots, H_\ell$ we  define the threshold $p(n;H_1,\dots, H_\ell,d)$ in the 
 natural way.
If $H_1=\dots =H_\ell=H$ we write $p(n;H_1,\dots, H_\ell,d)$ as $p(n; H,\ell,d)$.

\begin{exmp}
Casting Theorem \ref{thm:kreuterthresholds} into this notation, we have that $p(n;F,H,0)=n^{-1/m_K(F,H)}$, for a pair of graphs $F,H$ with $m_1(F) \leq m_1(H)$ (when $E(F)$ is nonempty and $H$ is not a matching) and this threshold is sharp.
\end{exmp}

Before we state our main result it is instructive to consider the following
result of Krivelevich, Sudakov and Tetali~\cite{kst}, which determines how many random edges need to be added to a dense graph to force the appearance of $H$ as a subgraph.  In our vertex-Ramsey framework, this corresponds to making the graph $(K_1, H)_v$-Ramsey. The corresponding threshold probability requires the following definition.
\begin{dfn} \label{def:m}
For a graph $H$, the appearance threshold for $H$ in the random graph $G(n,p)$ is determined by the parameter
 \[m(H):=\max\left\{\frac{e_J}{v_J}: J\subseteq H, v_J>0 \right\}. \]
Now, given any $k \in \mathbb N$, let
\[ m(H;k):=\min_{H_1\cup \dots\cup H_k=H;} \max _{i: H_i \neq \emptyset} m(H_i),\]
where the minimum is over all partitions of $H$ into $k$ induced subgraphs.\footnote{By a partition of $H$ into $k$ induced subgraphs, we mean there are $k$ (possibly  empty)
graphs $H_1, \dots, H_k$ such that each $H_i$ is an induced subgraph of $H$, the $H_i$ are all pairwise
vertex-disjoint, and $V(H)=V(H_1) \cup \dots \cup V(H_k)$.} 
\end{dfn}

Suppose that $G$ is a graph of density more than $1 - 1/(k-1)$ and we wish to find a copy of $H$ in $G \cup G(n,p)$. Informally, we partition $H$ into $k$ parts $H_1, \hdots, H_k$ that are as sparse as possible, with the idea being to use the (few) edges of $G(n,p)$ to build the parts $H_i$, and then find the edges between parts in the dense graph, thereby completing a copy of $H$.  Note that $m(H;k) = 0$ if and only if $\chi(H) \le k$, in which case we can partition $H$ into $k$ independent sets.  Then we do not require any random edges; the dense graph itself will already contain $H$.

\begin{thm}[Krivelevich,
Sudakov and Tetali~\cite{kst}]\label{thm:kst}
Let $0 < d < 1$ be  fixed  and let $k \geq 2$ be the unique integer
satisfying $1-1/(k-1) < d \leq 1 -1/k$.
Let $H$ be a graph with at least one edge. Then
\[p(n;K_1,H,d)=n^{-1/m(H;k)},\]
where here, and throughout, we follow the convention that $n^{-1/0}:=0$.
\end{thm}

Our main result (Theorem~\ref{thm:main} below) essentially resolves the $(H_1,H_2)_v$-Ramsey problem for randomly perturbed graphs for all pairs $(H_1,H_2)$ involving at least one clique. To state this result we define some notation capturing the probabilistic vertex Ramsey thresholds for all pairs of graphs.
\begin{dfn}
Given graphs $F$ and $H$, we write
\[ \beta(F,H) := \begin{cases}
    m_K(F,H) & \text{if } m_1(F) \le m_1(H), e(F) \ge 1, \\
    m_K(H,F) & \text{if } m_1(H) < m_1(F), e(H) \ge 1, \\
    m(H) & \text{if } e(F) = 0, e(H) \ge 1, \\
    m(F) & \text{if } e(H) = 0, e(F) \ge 1, \\
    0 & \text{if } e(F) = e(H) = 0.
    \end{cases}
\]
\end{dfn}
That is, $\beta(F,H)$ is defined so that $n^{-1/\beta(F,H)}$ is the threshold for the property that $G(n,p)$ is $(F,H)_v$-Ramsey.  We can now define our perturbed vertex Ramsey threshold, which is an extension of Definition~\ref{def:m}.
\begin{dfn} \label{def:mstar}
Given $r\in \mathbb N$, $k \geq 2$ and a graph $H$, define
$$m^*(K_r,H;k):= \max_{r_1+\dots +r_k\leq r-1;} \min_{H_1\cup \dots \cup H_k=H; } \max_{i: H_i \neq \emptyset}
\beta (K_{r_i+1}, H_i).$$
Here the first maximum is taken over all tuples $(r_1,\dots, r_k)$ of non-negative integers
that sum to at most $r-1$; the minimum is over all partitions of $H$ into $k$ induced 
subgraphs;
the final maximum is over all $i$ such that $H_i$ contains 
at least one vertex.
\end{dfn}

\begin{thm}\label{thm:main}
Let 
 $0 < d < 1$ be  fixed  and let $k \geq 2$ be the unique integer
satisfying $1-1/(k-1) < d \leq 1 -1/k$.
 Given any $r\in \mathbb N$ and any graph $H$ we have
$$p(n;K_r,H,d) = n^{-1/m^*(K_r, H;k)}.$$
\end{thm}
Note that Theorem~\ref{thm:main} is general in the sense that  it covers the full range
of densities $d \in (0,1)$ (not just small values of $d$). Further,
Theorem~\ref{thm:kst} is precisely the $r=1$ case of 
Theorem~\ref{thm:main}.

By computing the values of $m^* (K_s,K_t;2)$ one obtains from Theorem~\ref{thm:main} the following more explicit result for pairs of cliques, where for $a \leq b \in \mathbb{N}$, we define $\Psi(n;a,b):= n^{-\frac{1}{m_K(K_a,K_b)}}$.
\begin{cor}\label{thm:cliques}
Let $3\leq s\leq t$, $d\in(0,1/2)$. Then
\[
p(n;K_s,K_t,d)=\left\{ 
\begin{array}{llc}
   
   \Psi(n;t-1,t-1)=n^{-\frac{2}{t-1}} & \mbox{if }s=t; &(i)  \\
      \Psi(n;t-s,s) & \mbox{if } \frac{t+1}{2}\leq s <t; &(ii)\\
      \Psi(n;s,t/2) & \mbox{if } 3\leq s \leq t/2 \mbox{ and } t=2t' \mbox{ is even;} &(iii)\\
       \Psi(n;\floor{\frac{s+1}{2}},\frac{t+1}{2}) & \mbox{if } 3\leq s \leq (t-1)/2 \mbox{ and } t=2t'-1 \mbox{ is odd.} &(iv)
      
\end{array}\right.
\]
Moreover, these thresholds can be taken to be sharp whenever $s\notin\{t/2,t\}$. \qed
\end{cor} 
Although Theorem~\ref{thm:main} does not always guarantee sharp thresholds, by analysing its proof one can see that, in all cases except when $s \in \{t/2, t\}$, 
the threshold in Corollary~\ref{thm:cliques} is determined by Theorem~\ref{thm:kreuterthresholds}, which does provide a sharp result. Hence, we obtain the moreover part
of the corollary. On the other hand,  when $s \in \{t/2, t\}$, the threshold probability is not sharp. Indeed, this is since the threshold  comes from the appearance of a subgraph, which does not have a sharp threshold.\footnote{More precisely, consider the case when $s=t$. Let $G$ be the complete balanced bipartite graph on $n$ vertices with classes $A$ and $B$. Given any constant $C$, define $p=C \Psi(n;t-1,t-1)$. If we colour $A$ red and $B$ blue, $G \cup G(n,p)$ has a monochromatic copy of $K_t$ if and only if $G(n,p)[A]$ or $G(n,p)[B]$ contains $K_t$ as a subgraph. The probability of this occurring is bounded away from both $0$ and $1$, showing the lack of a sharp threshold. The case when $s=t/2$ is similar.}

Recall that in the random graph setting one needs $\Theta(n^{2-1/m_K(K_s,K_t)})$ random edges to ensure 
$G(n,p)$ is $(K_s,K_t)_v$-Ramsey. Corollary~\ref{thm:cliques} demonstrates that one needs far fewer random
edges 
to make any dense $n$-vertex graph $(K_s,K_t)_v$-Ramsey. However, the precise number of random edges depends
(in a rather subtle way) on arithmetic properties of the pair $(s,t)$.

\subsection{Some intuition for vertex Ramsey problems in randomly perturbed graphs}
In this section our aim is to convince the reader that the vertex Ramsey problem for
randomly perturbed graphs
is in general  more subtle than its counterpart in the random graph setting.

In Theorem~\ref{thm:onegraphthreshold} the threshold is \emph{universal} in the following sense:
the threshold for $G(n,p)$ being $(H,\ell)_v$-Ramsey is the point above which  every linear sized subset of
$G(n,p)$ w.h.p.\ contains a copy of $H$. It is easy to see that this property guarantees a
graph is $(H,\ell)_v$-Ramsey (as one of the colour classes in any vertex $\ell$-colouring will
have linear size).
Thus, crucially the `reason' for the location of the threshold is the \emph{same} for every graph $H$ (that is
not a matching). Moreover, this reason is independent of the number of colours used.

Similarly, the threshold in Theorem~\ref{thm:kreuterthresholds} is universal. Indeed, given \emph{any}
sequence of graphs $H_1,\ldots, H_\ell$ as in the theorem,
 the intuition behind
the threshold for the property of $G(n,p)$ being $(H_1,\dots, H_\ell)_v$-Ramsey is the same:
the threshold is the point at which the expected number of vertex-disjoint copies of $H_\ell$ is roughly
the same order of magnitude as the maximal order of a $H_{\ell-1}$-free subgraph of $G(n,p)$. (See the discussion
in~\cite{kreuter1996threshold}.) Again this threshold does not depend on the number $\ell$ of colours.

On the other hand, the intuition behind where the parameter $m^*(K_r,H;k)$
comes from is more involved than the intuition for 
Theorems~\ref{thm:onegraphthreshold} and~\ref{thm:kreuterthresholds}; we discuss this more
when proving Theorem~\ref{thm:main}. 
Moreover,
the threshold for the perturbed vertex Ramsey
problem can depend on the number of colours. Indeed,
we saw in Corollary~\ref{thm:cliques} that for every $ r\geq 3$ and $d\in (0,1/2)$
the number of random edges required to ensure an $n$-vertex graph of density $d$ is w.h.p.\ 
$(K_r,2)_v$-Ramsey 
is significantly smaller than $p(n;K_r,2,0)$; that is, 
significantly smaller than the number of edges needed to ensure
$G(n,p)$ is w.h.p.\ $(K_r,2)_v$-Ramsey. On the other hand, given any 
$\ell \geq 4$, we actually have that $p(n;K_{r},\ell,d)=p(n;K_{r},\ell,0)$ for all $d\in (0,1/2]$.
In fact, this phenomenon is part of a more general observation.

\begin{obs}\label{obs:multicol}
Let  $\ell \geq 4$ and $H$ be graph with at least one edge. Let $d\in (0,1/2]$. Then
$$p(n;H,\ell,d) = n^{-1/m_1(H)}=p(n;H,\ell,0).$$

Indeed, if $p \ge C n^{-1/m_1(H)}$ for some constant $C$, Theorem~\ref{thm:onegraphthreshold}
shows that $G(n,p)$ itself will be $(H,\ell)_v$-Ramsey w.h.p., and hence this is an upper bound on the
perturbed vertex Ramsey threshold.

For the lower bound, take $G$ to be a complete balanced bipartite $n$-vertex 
graph with vertex classes $A$ and $B$.  
By Theorem~\ref{thm:onegraphthreshold}, if $p \le c n^{-1/m_1(H)}$ for some constant $c$, then with high 
probability both $G(n,p)[A]$ and $G(n,p)[B]$ are not $(H,2)_v$-Ramsey.  This therefore
implies there exists a  $4$-colouring of the
vertices of $G \cup G(n,p)$ without a monochromatic copy of $H$.
\end{obs}
Given $\ell,k \in \mathbb N$ with $\ell\geq 2k$, and any $d\in (0,1-1/k]$, notice that by considering
the complete balanced $k$-partite $n$-vertex graph, one can similarly conclude that
$p(n;H,\ell,d) =p(n;H,\ell,0)$ for any graph $H$ with at least one edge.

\subsection{Notation}
Throughout the paper we omit floors and ceilings whenever this
does not affect the argument. Further,
we use standard graph theory and asymptotic notation. In particular, for a graph $G$, $v(G)$ denotes the number of vertices of $G$ and $e(G)$ the number of edges of $G$;
note that we often condense this notation
to $v_G$ and $e_G$ respectively. 
We say a graph is \emph{nonempty} if it has a nonempty vertex set, and unless otherwise specified, we shall take the vertex set to be $[v_G] := \{1, 2, \hdots, v_G\}$.
Given a hypergraph $\bm{H}$ and a set $I\subset V(\bm{H})$, $\deg_{\bm{H}}(I)$ denotes the number of edges of $\bm{H}$ that contain the set $I$.

Given a set $X$ and $r \in \mathbb N$ we write
$\binom{X}{r}$ for the set of all subsets of
$X$ of size $r$. Similarly, if $V$ is a set of vertices and $F$ is a graph, we denote by $\binom{V}{F}$ the set of all possible copies of $F$ supported on vertices in $V$. Here we consider these  copies of $F$ to be distinct if they have distinct sets of edges, so $\card*{\binom{V}{F}} = \binom{\card{V}}{v_F} \frac{v_F!}{\aut(F)}$, where $\aut(F)$ is the number of automorphisms of $F$. We also use the notation $\binom{G}{F}$ to denote the set of copies of $F$ in a graph $G$.

\subsection{Organisation of the paper}
In Section~\ref{sec:proof} we prove Theorem~\ref{thm:main}, handling the $0$-statement in Section~\ref{sec:0-statement} and the $1$-statement in Section~\ref{sec:1-statement}. In the latter section, we shall require a `robust' version of the $1$-statement of Theorem~\ref{thm:kreuterthresholds} (Theorem~\ref{thm:robustKreuter1}), which we prove in Section~\ref{sec3}. Finally, in Section~\ref{sec:conc} we give some concluding remarks.

\section{Proof of the perturbed threshold} \label{sec:proof}

In this section we prove Theorem~\ref{thm:main}.  We present the arguments for the $0$- and $1$-statements in separate subsections below.

\subsection{The 0-statement} \label{sec:0-statement}

Here we will show the existence of a graph $G$ of density at least $d$ such that, when $p = o \left( n^{-1/m^*(K_r, H;k)} \right)$, with high probability the vertices of $G \cup G(n,p)$ can be two-coloured without a red copy of $K_r$ or a blue copy of $H$.

\subsubsection{Kreuter's Theorem for families} \label{sec:families}

In showing the existence of a good colouring, we will use the $0$-statement of Theorem~\ref{thm:kreuterthresholds}, which shows that $G(n,p)$ can be vertex-coloured while avoiding monochromatic subgraphs.  However, in our application, we will have to avoid several subgraphs in the same colour class, and therefore need the following generalisation of Theorem~\ref{thm:kreuterthresholds} to families of graphs.

\begin{prop} \label{prop:0statementforfams}
Let $\cF$ and $\cH$ be two finite families of nonempty graphs, and let 
\[m=m_K(\cF,\cH):=\min_{F\in \cF, H\in \cH} \beta(F,H).\]
If $p = o \left( n^{-1/m} \right)$, then with high probability there is a red/blue-colouring of the vertices of $G(n,p)$ without a red copy of any graph $F\in \cF$ and without a blue copy of any $H\in \cH$. 
\end{prop}

We remark that Proposition~\ref{prop:0statementforfams} is tight, since if $p = \omega \left(n^{-1/m} \right)$, then with high probability $G(n,p)$ is $(F,H)_v$-Ramsey by Theorem~\ref{thm:kreuterthresholds}, where $(F,H)$ is the minimising pair in the definition of $m$. 

The proof of Proposition~\ref{prop:0statementforfams} is nearly identical to the proof of the $0$-statement of Theorem~\ref{thm:kreuterthresholds}; that is, the case when both $\mathcal F$ and $\mathcal H$
each contain a single graph.
We therefore simply sketch the key idea here and refer the reader to \cite{kreuter1996threshold} for the details.

\smallskip

First we handle the degenerate cases.  Suppose one of the families, say $\mc F$, contains a graph $F$ with no edges. The parameter $\beta(F,H)$ is then the appearance threshold for the graph $H$ in $G(n,p)$, and so if $p = o \left(n^{-1/m}\right)$, we have that with high probability $G(n,p)$ has no copy of any graph from $\mc H$, and so we can colour all its vertices blue.  The other degenerate case is when all graphs in $\mc F$ and $\mc H$ have edges, but both families contain matchings, say $F$ and $H$.  In this case, $\beta(F,H) = 1$.  If $p = o \left( n^{-1} \right)$, then with high probability $G(n,p)$ is bipartite.  We can thus two-colour its vertices such that each colour class is an independent set, and thus has no copy of any graph from $\mc F$ or $\mc H$.

We may therefore assume that every graph in $\mc F$ has edges, and that $\mc H$ does not contain a matching, bringing us to the setting of Theorem~\ref{thm:kreuterthresholds}. The proof now follows a similar scheme to other $0$-statement proofs in random Ramsey settings, e.g.\ \cite{kohayakawa1997threshold}. One begins by supposing for a contradiction that you cannot red/blue-colour the graph $G(n,p)$ avoiding  red copies of graphs in $\cF$ and blue copies of graphs in $\cH$. Using this fact one can define some set $\cG$ of graphs obtained by `gluing together' copies of graphs in $\cF$ and in $\cH$ in certain ways, and show that $G(n,p)$ must contain a graph in $\cG$. In \cite{kreuter1996threshold}  $\cG$ is defined by way of an algorithm that finds a copy of some $G\in \cG$ in $G(n,p)$. In order to do this they pass to a `critical' subgraph $G'$ of $G(n,p)$ which is minimal (in terms of copies of graphs in $\cF$ and $\cH$) with respect to the property of not being able to $2$-colour $G'$ avoiding red copies of graphs in $\cF$ and blue copies of graphs in $\cH$. In $G'$, one can see that for every copy $T$ of a graph in $\cF$ or $\cH$ and every vertex $v$ of $T$, there is a copy of a graph in the other family which intersects $T$ exactly at the vertex $v$ \cite{kreuter1996threshold}*{Claim 1}. The algorithm \cite{kreuter1996threshold}*{Procedure Hypertree} which builds a subgraph $J$ of $G'$ is then defined by repeatedly adding copies of some $F\in \cF$ or $H\in \cH$, so that the copy intersects the previous copy in exactly one vertex. The proof  then works by analysing this procedure and the graphs in $\cG$ that can be found using this procedure. In particular, one keeps track of a function $f(i)$ which controls the exponent of the expected number of $J_i$, where $J_i$ is the graph found in $G'$ after $i$ steps of the algorithm. The procedure will stop if the $f(i)$ gets too small or if the procedure continues for roughly $\log n$ steps.  This will lead to a contradiction, as the graphs in $\cG$ which are the possible outcomes of this procedure are all dense graphs and are either large or satisfy a very strong density condition \cite{kreuter1996threshold}*{Claim 5} and hence are very unlikely to occur in $G(n,p)$ at this density. One can also bound the size of $\cG$ \cite{kreuter1996threshold}*{Claim 6} so that a union bound will guarantee that with high probability no such graph in $\cG$ is found in $G(n,p)$. In the calculations involved in the analysis \cite{kreuter1996threshold}*{Claims 2,3 and 4} of the effect on $f$ in each step of the algorithm, there are some minor changes in our setting as we have to consider the possibility of any member of our family being added by the algorithm to update the $J_i$. However, it can be seen that adding a `denser' graph than the graph which is in the minimising pair for the definition of $m$ will only help the situation, in that the function $f$ can only decrease further, meaning that the resulting $J_i$ is at most as likely to appear in $G(n,p)$ as the $J_i$ obtained by adding the minimising graph as in the calculations in \cite{kreuter1996threshold}.

\subsubsection{Proof of the 0-statement}

We may assume that $m^*(K_r, H; k) > 0$, as otherwise there is no $0$-statement to prove. We take the dense graph $G_n$ to be the balanced complete $k$-partite graph on $n$ vertices, which has density at least $1 - \tfrac{1}{k}$, and let $V_1, V_2, \hdots, V_k$ be the $k$ vertex classes.

Let $(r_1, \hdots, r_k)$ be the maximising vector in Definition~\ref{def:mstar}.  For each $i \in [k]$, we define a family of nonempty subgraphs of $H$ by
\[ \mc H_i = \left\{ H' : \emptyset \neq H' \subseteq H, \, \beta(K_{r_i+1}, H') \ge m^*(K_r, H; k) \right\}. \]

\begin{clm}
For each $i$, $H \in \mc H_i$.
\end{clm}

\begin{proof}
Suppose $H \notin \mc H_j$ for some $j \in [k]$, and consider the partition $H = H_1 \cup H_2 \cup \hdots \cup H_k$, where $H_i = H$ if $i = j$ and $H_i = \emptyset$ otherwise.  We then have 
\[ \max_{i : H_i \neq \emptyset} \beta(K_{r_i + 1}, H_i) = \beta(K_{r_j + 1}, H) < m^*(K_r, H; k), \]
which contradicts the definition of $m^*(K_r, H; k)$.
\end{proof}

In particular, each of the families $\mc H_i$ is nonempty.  We can now describe, for each $i \in [k]$, our colouring of the vertices in $V_i$. Let $\mc F = \{ K_{r_i + 1} \}$ and $\mc H = \mc H_i$. By definition of $\mc H_i$, we have $\min\limits_{F \in \mc F, H' \in \mc H_i} \beta(F, H') \ge m^*(K_r, H; k)$. Since $p= o\left(n^{-1/m^*(K_r, H; k)} \right)=o\left(|V_i|^{-1/m^*(K_r, H; k)} \right)$, it follows from Proposition~\ref{prop:0statementforfams} that with high probability we can colour $V_i$ such that $G(n,p)[V_i]$ has neither a red $K_{r_i+1}$ nor a blue graph from $\mc H_i$.  The following claim shows this gives a valid colouring of $G_n \cup G(n,p)$, completing the proof of the $0$-statement.

\begin{clm}
With this colouring, $G_n \cup G(n,p)$ has neither a red $K_r$ nor a blue $H$.
\end{clm}

\begin{proof}
By construction, the largest red clique in $V_i$ is of order at most $r_i$.  The largest red clique in $V(G_n \cup G(n,p)) = \cup_i V_i$ therefore has at most $\sum_i r_i \le r-1$ vertices, and hence the colouring is red-$K_r$-free.

Suppose there was a blue copy of $H$, and let $H = H_1 \cup \hdots \cup H_k$ be the partition of $H$ induced by the parts $V_i$. By definition of $m^*(K_r, H; k)$, there is some part $i \in [k]$ with $\beta(K_{r_i + 1}, H_i) \ge m^*(K_r, H; k)$ (and $H_i \neq \emptyset$). It then follows that $H_i \in \mc H_i$, but our colouring of $G(n,p)[V_i]$ avoids blue copies of any graph in $\mc H_i$, contradicting the existence of this blue copy of $H$.
\end{proof}

\subsection{The 1-statement} \label{sec:1-statement}

To prove the $1$-statement of Theorem~\ref{thm:main}, we need to show that whenever $p = \omega \left(n^{-1/m^*(K_r, H; k)} \right)$, $G_n \cup G(n,p)$ will with high probability be $(K_r, H)_v$-Ramsey. When $G_n$ is the complete $k$-partite graph, as it was in the proof of the $0$-statement, this amounts to finding the sparse parts of the graphs in $G(n,p)[V_i]$, which can then be joined together since we have a complete $k$-partite graph.

However, in our more general setting, $G_n$ is an arbitrary graph of density $d > 1 - 1/(k-1)$.  By employing Szemer\'edi's Regularity Lemma~\cite{szemeredi}, we shall find some structure in $G_n$ that mimics the behaviour of a complete $k$-partite graph.  These structural results, together with probabilistic tools concerning the random graph $G(n,p)$, are collected in the following subsections, before being used in the proof of Theorem~\ref{thm:main} in Sections~\ref{sec:1statementsketch} and~\ref{sec:1statementproof}.

\subsubsection{Structure in dense graphs} \label{sec:regularity}

Our application of the Regularity Lemma follows the standard lines.  We present here the necessary definitions and properties of regular pairs, referring the reader to the survey of Koml\'os and Simonovits~\cite{komsim} for further details.

\begin{dfn}
Given $\eps > 0$, a graph $G$ and two disjoint vertex sets $A, B \subset V(G)$, the pair $(A,B)$ is $\eps$-regular if for every $X \subseteq A$ and $Y \subseteq B$ with $\card{X} > \eps \card{A}$ and $\card{Y} > \eps \card{B}$, we have $\card{d(X,Y) - d(A,B)} < \eps$, where $d(S,T) := e(S,T)/(\card{S}\card{T})$ for any vertex sets $S$ and $T$. 
\end{dfn}

In essence, the edges between a regular pair `look random', in the sense that they are very well distributed.  The next lemma showcases some beneficial properties of these regular pairs: small sets of vertices typically have many common neighbours, and subsets of regular pairs inherit a large degree of regularity.  We omit the proofs of these facts, which can be found in~\cite{komsim}.

\begin{lem} \label{lem:regproperties}
Let $(A,B)$ be an $\eps$-regular pair in a graph $G$ with $d(A,B) = d$.
\begin{enumerate}[label=(\roman*)]
    \item If $\ell \ge 1$ and $(d - \eps)^{\ell-1} > \eps$, then
    \[ \card*{\left\{ (x_1, \hdots, x_\ell) \in A^{\ell} : \card{\cap_i N(x_i) \cap B} \le (d - \eps)^\ell \card{B} \right\} } \le \ell \eps \card{A}^\ell. \]
    \item If $\gamma > \eps$, and $A' \subset A$ and $B' \subset B$ satisfy $\card{A'} \ge \gamma \card{A}$ and $\card{B'} \ge \gamma \card{B}$, then $(A', B')$ is an $\eps'$-regular pair of density $d'$, where $\eps' := \max \{ \eps/ \gamma, 2 \eps \}$ and $\card{d' - d} < \eps$.
\end{enumerate}
\end{lem}

Szemer\'edi's Regularity Lemma then famously asserts that the vertices of any sufficiently large graph can be partitioned into a large but bounded number of parts, such that almost all pairs of parts are $\eps$-regular. We shall not require the full strength of the Regularity Lemma, but only the following corollary, which follows in combination with Tur\'an's Theorem~\cite{turan}.

\begin{prop} \label{prop:regularity}
For every $k \ge 2$ and $\alpha, \eps > 0$ with $\alpha \ge 6 \eps$, there is some $\eta := \eta(k, \alpha, \eps) > 0$ and $n_0 := n_0(k, \alpha, \eps)$ such that, if $n \ge n_0$ and $G$ is an $n$-vertex graph of density at least $1 - 1/(k-1) + 2 \alpha$, then there are pairwise disjoint vertex sets $V_1, \hdots, V_k \subset V(G)$ with $\card{V_1} = \hdots = \card{V_k} \ge \eta n$ such that, for each $1 \le i < j \le k$, the pair $(V_i, V_j)$ is $\eps$-regular of density at least $\alpha$.
\end{prop}

\subsubsection{Probabilistic tools} \label{sec:probtools}
While Proposition~\ref{prop:regularity} gives us the desired structure in the dense graph, we also require a couple of results about the random graph.  The first of these counts the number of copies of a fixed graph $H$ in $G(n,p)$. Following~\cite{randomgraphbook}, we define the following parameter for  $H$ and $p=p(n)$,
\begin{equation} \label{eqn:Phidefn} \Phi(H,p)=\Phi_{H,p}:=\min_{J\subseteq H, e_J>0 }n^{v_J}p^{e_J}.
\end{equation}
The lemma below shows that we are very unlikely to have significantly fewer copies of $H$ than expected.

\begin{lem}[Janson's inequality]  \label{lem:vanillajanson}

 Let $H$ be a nonempty graph, $p=p(n)$ and $\pzc{H}\subseteq \binom{[n]}{H}$ be some family of $\Omega(n^{v_H})$  potential copies of $H$ on $[n]$. Letting $X$ be the random variable that counts the number copies of $H$ in $\pzc{H}$  which appear in $G(n,p)$, we have that \[\mathbb{P}[X\leq 3\mathbb{E}[X]/4]\leq \exp (-\Omega(\Phi_{H,p})). \]
 \end{lem}
 
 The proof of this lemma
 follows almost immediately from the main result of \cite{janson1990poisson} (see also 
 \cite{randomgraphbook}*{Theorem 2.14}). Indeed, for each potential copy $S\in \pzc{H}$ of 
 $H$, let $X_S$ be the indicator random variable for the event that $S$ appears in 
 $G(n,p)$. Then $X=\sum_{S\in\pzc{H}}X_S$ and  \cite{randomgraphbook}*{Theorem 2.14} implies
 that \[\mathbb{P}[X\leq3 \mathbb{E}[X]/4] \leq \exp 
 \left(-\frac{\mathbb{E}[X]^2}{32\Delta}\right),\]
 with \[\Delta:=\sum_{(S,S')\in \pzc{H}^2:E(S)\cap E(S')\neq \emptyset}\mathbb{E}[X_S X_{S'}].\]
 Lemma \ref{lem:vanillajanson} then follows upon noticing that $\Delta=O\left(\mathbb{E}[X]^2/\Phi_{H,p}\right)$ as done, for example,  in \cite{randomgraphbook}*{Theorem 3.9}.
 
Second, we shall make use of the vertex Ramsey properties of $G(n,p)$. However, a couple of complications arise in our application. We shall need the monochromatic subgraphs we find to interact well with the deterministic base graph, and shall therefore require them to be suitably well-located. Moreover, we will need to appeal to the Ramsey properties of the random graph over various vertex subsets. The following theorem is thus a `robust' version of the $1$-statement of Theorem~\ref{thm:kreuterthresholds}, which guarantees the existence of `good' monochromatic copies of subgraphs, and shows that the Ramsey properties hold with sufficiently high probability to be applied several times.

\begin{thm}\label{thm:robustKreuter1}
Let  $F$ and $H$ be graphs with $0 < m_1(F) \le m_1(H)$. Then there exist $\delta_0, c>0$ such that for all $0<\delta<\delta_0$, $t=t(n)\leq \exp(n^c)$ and $\eta_0>0$,
there exists a $C>0$ such that the following holds.
Suppose that \[\bfc{U}=\left\{(U_i,\pzc{F}_i,\pzc{H}_i):i\in[t] \right\}\]
is a collection of triples such that for each $i$, $U_i\subseteq [n]$ with $|U_i|\geq \eta_0 n$, $\pzc{F}_i\subseteq\binom{U_i}{F}$ with $\card*{\binom{U_i}{F} \setminus \pzc{F}_i} \leq \delta |U_i|^{v_F}$ and $\pzc{H}_i\subseteq \binom{U_i}{H}$ with $\card*{\binom{U_i}{H} \setminus \pzc{H}_i}\leq \delta |U_i|^{v_H}$. Then, if $p\geq Cn^{-1/m_K(F,H)}$, the following holds with high probability in $G(n,p)$.  For any two-colouring of $[n]$ and every $i\in[t]$, there is either a red copy $S\in \pzc{F}_i$ of $F$ or a blue copy $T\in\pzc{H}_i$ of $H$. 
\end{thm}
Notice that, unlike in
Theorem~\ref{thm:kreuterthresholds},
Theorem~\ref{thm:robustKreuter1} allows for 
both $F$ and $H$ to be matchings.
The proof is similar to that of the $1$-statement of Theorem~\ref{thm:kreuterthresholds}, but this strengthened version requires a few additional ideas and some careful analysis of the failure probabilities at each step. We defer these details until Section~\ref{sec:robustKreuter1statement}, and instead complete the proof of Theorem~\ref{thm:main} next.

\subsubsection{An algorithm for the 1-statement} \label{sec:1statementsketch}

Recall that to prove the 1-statement of Theorem~\ref{thm:main}, we need to show that whenever $G_n$ is an $n$-vertex graph of density $d > 1 - 1/(k-1)$, and $p = \omega\left(n^{-1/m^*(K_r,H;k)}\right)$, then with high probability $G(n,p)$ is such that every vertex colouring of $G_n \cup G(n,p)$ gives rise to a red copy of $K_r$ or a blue copy of $H$. To do this, we use an algorithm that, given a vertex colouring of $G_n \cup G(n,p)$, returns one of the desired monochromatic subgraphs. We motivate and describe the algorithm in this subsection, while in the next we prove that, with high probability, $G(n,p)$ is such that the algorithm succeeds for every vertex colouring.

By Proposition~\ref{prop:regularity}, we can find a $k$-tuple of vertex sets $V_1, V_2, \hdots, V_k$ such that each pair is $\eps$-regular and reasonably dense. We then aim to use the Ramsey properties of $G(n,p)[V_i]$ to find suitable monochromatic subgraphs that can be pieced together to form a red $K_r$ or a blue $H$.

However, a na\"ive application of Theorem~\ref{thm:kreuterthresholds} will not work. Indeed, by definition of $m^*(K_r, H; k)$, there is some vector $(r_1, \hdots, r_k)$ with $\sum_i r_i \le r-1$ and some partition $H = H_1 \cup \hdots \cup H_k$ such that $p = \omega ( n^{-1/\beta(K_{r_i + 1}, H_i)} )$ for all $i$ with $H_i \neq \emptyset$. We can therefore expect that, for each $i$, we find a red $K_{r_i + 1}$ or a blue $H_i$ in any vertex colouring of $G(n,p)[V_i]$.

If these monochromatic subgraphs were all of the same colour, then we could hope to combine them to form a red clique (which would in fact be of size $r + k - 1$, significantly larger than required) or a blue copy of $H$. However, we could well find red cliques in some parts and blue subgraphs in others, which would leave us unable to complete either of the desired graphs.

Instead, we must use the full power of Definition~\ref{def:mstar}, which provides a suitable partition $H = H_1 \cup \hdots \cup H_k$ not just for \emph{some} vector $(r_1, \hdots, r_k)$, but rather for \emph{all} vectors $(r_1, \hdots, r_k)$ satisfying $\sum_i r_i \le r-1$.  We shall therefore proceed in stages, incrementally either increasing the size of a red clique or finding the next piece needed for a blue copy of $H$. We let $r_i$ denote the size of the largest red clique we have found in $V_i$ thus far, starting with $\vec{r} = \vec{0}$.

Given the current vector $\vec{r}$, we let $H = H_1 \cup \hdots \cup H_k$ be the corresponding minimising partition of $H$.  We then go through the parts in turn, applying the $(K_{r_i + 1}, H_i)_v$-Ramsey property of $G(n,p)[V_i]$ to find a blue $H_i$ or a red $K_{r_i+1}$. In the former case, we proceed to the next part.  If we make it through each of the $k$ parts, we will have found all the parts $H_i$ needed to build a blue copy of $H$.

Otherwise, in the latter case, we have increased the size of our red clique. We then update the vector $\vec{r}$ and the corresponding partition of $H$, return to the first part $V_1$, and resume the process.  Since this increases the size of our red clique, we will have built a red $K_r$ if this latter case occurs $r$ times.

There are still technicalities that need to be dealt with --- for instance, to ensure we can combine the monochromatic structures we find, we will need to restrict ourselves to the common neighbourhoods of the parts we have already found. This further requires us to only consider subgraphs with many common neighbours in all other parts, which is why we need the more robust $1$-statement of Theorem~\ref{thm:robustKreuter1}. In the following subsection, we provide the formal details of this algorithm and use the tools we have collected to prove that it runs successfully.

\subsubsection{Proof of correctness} \label{sec:1statementproof}

Given $\alpha > 0$ and $p = \omega (n^{-1/m^*(K_r, H; k)})$, our goal is to show that for any $n$-vertex graph $G_n$ of density $d \ge 1 - 1/(k-1) + 2 \alpha$, the graph $G_n \cup G(n,p)$ is with high probability $(K_r,H)_v$-Ramsey. Applying Proposition~\ref{prop:regularity} to $G_n$ with some suitably small\footnote{For our purposes, it suffices to take $\eps = \frac{\delta'}{4 k^2 r v_H (r + 2^{v_H})} \left( \frac{\alpha}{2} \right)^{2 k r v_H (r + 2^{v_H})}$, where $\delta'$ is the minimum value of $\delta_0$ from Theorem~\ref{thm:robustKreuter1} when the graph $F$ is a clique on at most $r$ vertices and the graph $H$ in the theorem is a subgraph of our given graph $H$.} regularity parameter $\eps$ gives $k$ pairwise-disjoint vertex sets $V_1, V_2, \hdots, V_k$, such that each pair $(V_i, V_j)$ is $\eps$-regular of density at least $\alpha$. We initiate by setting $U_j=V_j$ for $j\in [k]$. The sets $U_i$ will keep track of where we look to find certain subgraphs.

At several stages in the algorithm, we will, for some $i$, find a 
(constant sized) subgraph $\Gamma \subset G(n,p)[U_i]$, and will 
then want to shrink all the other parts $U_j$ to the common neighbours in 
$G_n$ of the vertices of $\Gamma$.  We shall therefore call $\Gamma$ \emph{popular} (with respect to some choice of $U_j\subseteq V_j$ for $j\in[k]$) if its vertices have at least $(\tfrac{\alpha}{2})^{v_{\Gamma}}|U_j|$ common $G_n$-neighbours in each $U_j$, $j \neq i$. 
Lemma~\ref{lem:regproperties} ensures that most potential copies of $\Gamma$ will be popular, and that when we shrink the sets $U_j$ to their large common neighbourhoods, the pairs will remain $\eps'$-regular with density at least $\tfrac{\alpha'}{2}$. By choosing the initial value of $\eps$ small enough, we can ensure that all subsequent values of $\eps'$ remain small.

\medskip

We first find copies of the subgraphs of $K_r$ and $H$ that are likely to have appeared in $G(n,p)$. Let $t := \max \{ s \in [r] : m(K_s) \le m^*(K_r, H; k) \}$ and let $\mc G := \{ H[U] : U \subseteq V(H), m(H[U]) \le m^*(K_r, H; k) \}$. We then define the graph $\Gamma$ to be the disjoint union of $v_H$ copies of $K_t$ together with one copy of each graph in $\mc G$. Then, for each $i \in [k]$ in turn, find a popular copy $\Gamma_i$ of $\Gamma$ in $G(n,p)[U_i]$, and shrink all other parts $U_j$ to the common neighbours of $V(\Gamma_i)$ in $U_j$. Note that, at the end of this process, for all $i$ the graph $\Gamma_i$ remains in the set $U_i$.

\medskip

We can now start the procedure outlined in the previous subsection.  Given an arbitrary red/blue colouring of the vertices of $G_n \cup G(n,p)$, we shall denote by $R_i$ the largest red clique found in $G(n,p)[U_i]$ thus far, initially setting $R_i = \emptyset$ for all $i \in [k]$.  The vector $\vec{r}$ will be defined by $r_i := v_{R_i}$, so we begin with $\vec{r} = \vec{0}$.

The outer loop of the algorithm runs as long as $\sum_i r_i \le r-1$, which means we have not yet found a red $K_r$. In this case, we take the minimising partition $H = H_1 \cup \hdots \cup H_k$ for the vector $\vec{r}$ in Definition~\ref{def:mstar}, and try to find a blue copy of $H$ according to this partition.

The inner loop of the algorithm runs over $i \in [k]$.  If $H_i = \emptyset$, then there is nothing to find in $G(n,p)[U_i]$, and so we proceed to the next part. Otherwise, since, as we shall soon show, $G(n,p)[U_i]$ is robustly $(K_{r_i+1}, H_i)_v$-Ramsey, we will find a popular blue $H_i$ or a popular red $K_{r_i+1}$. If we have a blue $H_i$, we let $B_i$ be this copy of $H_i$, shrink all other parts $U_j$ to the common neighbours of $V(B_i)$, and then proceed to the next part.

On the other hand, if we find a red $K_{r_i+1}$, then we have increased the size of our red clique. We then set $R_i$ to be this larger clique and shrink all other parts $U_j$ to the common neighbours of $V(R_i)$.  We update the vector $\vec{r}$, replacing $r_i$ with $r_i + 1$, and then break the inner loop and proceed to the next iteration of the outer loop (trying to find the new optimal partition of $H$, starting in $U_1$).

\medskip

Since we shrink to common neighbourhoods at each step, we ensure that the pieces we find in $G(n,p)[V_i]$ can be combined to form the graphs we need in $G_n \cup G(n,p)$.  In particular, if the inner loop were to run through all $k$ steps, then $\cup_i B_i$ would give a blue copy of $H$.  On the other hand, each iteration of the outer loop increases the size of our red clique, and after $r$ iterations $\cup_i R_i$ would give a red $K_r$.  Thus, after finitely many steps, the algorithm must return either a blue $H$ or a red $K_r$. Moreover, since the running time of the algorithm is bounded, the suitably small $\varepsilon$ we required at the beginning is some constant depending on $r, k$ and $H$, independent of the actual course taken by the algorithm.

\medskip

To complete the proof, we need to show that with high probability, $G(n,p)$ is such that the algorithm succeeds in finding all the necessary subgraphs at each step of the algorithm.  Firstly, let us consider the subgraphs $\Gamma_i$ which we find at the beginning of the algorithm. If $\Gamma$ is an independent set, clearly one can find these popular copies. Otherwise, noting that
$m(\Gamma) \leq m^*(K_r, H; k)$,
Lemma~\ref{lem:vanillajanson} ensures that we find these popular copies with high probability. 

Now we show that with high probability $G(n,p)[U_i]$ will always be robustly $(K_{r_i+1}, H_i)_v$-Ramsey, which we will mostly achieve through use of Theorem~\ref{thm:robustKreuter1}.   However, this only applies when $r_i, e(H_i) \ge 1$.  For the degenerate cases, we will need to make use of the graphs $\Gamma_i$ we found at the beginning. Suppose first that $r_i = 0$. By definition, we have $m^*(K_r, H;k) \ge \beta(K_{r_i+1}, H_i) = \beta(K_1, H_i) = m(H_i)$, and so $H_i$ appears in $\Gamma_i$. Then either this copy of $H_i$ is completely blue, or we find a red $K_1$, and so $\Gamma_i \subseteq G(n,p)[U_i]$ is indeed $(K_1, H_i)_v$-Ramsey. The other case, when $e(H_i) = 0$, follows similarly.  This time we have $m^*(K_r, H; k) \ge \beta(K_{r_i + 1}, H_i) = m(K_{r_i+1})$, and so $\Gamma_i$ contains $v_H$ copies of $K_{r_i + 1}$.  Either one of them is completely red, in which case we are done, or we have $v_H$ blue vertices, which in particular gives a blue copy of $H_i$.

This leaves us with the case when both $K_{r_i+1}$ and $H_i$ have edges. Again, by definition, we have $m^*(K_r, H; k) \ge \beta(K_{r_i+1}, H_i)$.  Thus, since $p = \omega(n^{-1/m^*(K_r,H;k)})$, Theorem~\ref{thm:kreuterthresholds} shows we should expect $G(n,p)[U_i]$ to be $(K_{r_i+1}, H_i)_v$-Ramsey. However, we need this to be true for all sets $U_i$ that could arise, and also need to find popular monochromatic copies of $K_{r_i+1}$ or $H_i$, and thus we apply Theorem~\ref{thm:robustKreuter1} instead.

Note that there is some constant $s = s(r,H;k)$ such that the sets $U_i$ that arise are the common neighbourhoods of a set of at most $s$ vertices, and hence there are at most $n^s$ many possibilities for the sequence $(U_i : i \in [k])$. In particular, this is far fewer than the $\exp(n^c)$ allowed by Theorem~\ref{thm:robustKreuter1}. Moreover, as these are always neighbourhoods of popular subgraphs, there is some constant $\eta_0 > 0$ such that $\card{U_i} \ge \eta_0 n$.

Thus, given a set $\Gamma \subseteq \bigcup_j V_j$ of at most $s$ vertices, we let, for each $j \in [k]$, $U_j \subseteq V_j$ be the common neighbours (in $G_n$) of $\Gamma \setminus V_j$. Provided $|U_j| \ge \eta_0 n$ for each $j$, we then define, for every $i \in [k]$, a triple $(U', \mc F', \mc H') \in \bfc{U}$, where we take $U'=U_i$, we let $\mc F'$ be all possible popular (with respect to the $U_j$) copies of $K_{r_i+1}$ in $U_i$, and let $\mc H'$ be all possible popular copies of $H_i$ in $U_i$. Lemma~\ref{lem:regproperties} (and our choice of small $\eps$) ensures that $\card*{\binom{U_i}{K_{r_i + 1}} \setminus \mc F'} \le \tfrac12 \delta_0 \card{U_i}^{r_i+1}$ and $\card*{\binom{U_i}{H_i} \setminus \mc H'} \le \tfrac12 \delta_0 \card{U_i}^{v_{H_i}}$.

We therefore satisfy all the requirements of Theorem~\ref{thm:robustKreuter1}, and can conclude that with high probability, the random graph $G(n,p)$ has the property that whenever we require $G(n,p)[U_i]$ to be $(K_{r_i+1}, H_i)_v$-Ramsey, it will be. As there are only finitely many pairs $(K_{r_i+1}, H_i)$ to consider, it follows that the algorithm succeeds with high probability overall, completing the proof.

\section{Robust Ramsey properties of random graphs} \label{sec3}
The aim of this section is to give a proof of Theorem \ref{thm:robustKreuter1}. Although our proof here is similar to that of Kreuter~\cite{kreuter1996threshold}, we choose to give the details as the argument is somewhat delicate and our proof departs from the original in some key steps. In particular, instead of using Tur\'an's theorem to estimate the maximal size of a set of vertex-disjoint copies of a given graph (as  done by Kreuter~\cite{kreuter1996threshold}), we use a probabilistic approach (as in \cite{alonspencer}*{Lemma 7.3.1}) which allows us to analyse the relevant subgraph counts at every step of the proof and guarantee that we find monochromatic copies of the graphs on the desired vertex sets. We first give some probabilistic tools and intermediate lemmas before embarking on the proof. 
\subsection{Probabilistic tools}

 \subsubsection{Chebyshev's inequality}
 
We will use the following well known inequality, see e.g.\ \cite{alonspencer}*{Chapter 4}. 

\begin{lem} \label{lem:chebyshev} 
 Suppose $\{A_i:i\in I\}$ is a finite set of events in some probability space and for each $i\in I$, let $X_i$ be the indicator random variable for the event $A_i$. Write $i\sim j$ if the events $A_i$ and $A_j$ are \emph{not} independent. Further, let $X:=\sum_{i\in I} X_i$ be the sum of the indicator random variables and define \[\Delta:=\sum_{i\sim j}\mathbb{E}[X_i X_j],\]
 where the sum is over all ordered pairs $(i,j)$ (including diagonal terms). Then for all $t>0$, \[\mathbb{P}[|X-\mathbb{E}[X]|\geq t]\leq \frac{\Delta}{t^2}.\]
 \end{lem}
 
\subsubsection{Janson's inequality for a refined random graph} \label{sec:janson}
We are primarily concerned with the appearance of a subgraph $H$ in a random graph $G(n,p)$. While it is easy to compute the probability of a given copy of $H$ being present in $G(n,p)$, it is inconvenient that these copies need not be independent --- two copies that share edges will be positively correlated.

To smooth the analysis, we will artificially introduce a filtering stage, where we select each copy of $H$ independently, and then only focus on the selected copies of $H$ that appear in $G(n,p)$. Formally, let~$\rho(H;q):~\binom{[n]}{H}\rightarrow~\{0,1\}$ be a function that randomly assigns $1$ with probability $q$  and $0$ with probability $1-q$ to each copy of $H$ in $K_n$, independently of the other choices. Recalling that $\binom{G(n,p)}{H}$ denotes the family of all copies of $H$ appearing in $G(n,p)$, we let $\binom{G(n,p)}{H}_q$ denote the random subfamily consisting of selected copies of $H$ in $G(n,p)$.

The following lemma shows that even when we are only interested in some dense collection of `good' copies of $H$, we are still very likely to have many of them appear in $\binom{G(n,p)}{H}_q$.
 
 \begin{lem}
 \label{lem:janson} 
 Let $H$ be a graph with at least one edge, $p=p(n), q=q(n)$, and $\pzc{H}\subseteq \binom{[n]}{H}$ be some family of $\Omega(n^{v_H})$  potential copies of $H$ on $[n]$. Letting $X_q := \card{\pzc{H} \cap \binom{G(n,p)}{H}_q}$, we have that \[\mathbb{P}[X_q\leq \mathbb{E}[X_q]/2]\leq \exp (-\Omega(\Phi_{H,p}))+\exp(-\Omega(qn^{v_H}p^{e_H})), \]
 where $\Phi_{H,p}$ is as defined in~\eqref{eqn:Phidefn}.
 \end{lem}
 The lemma follows from an application of Lemma~\ref{lem:vanillajanson}, which gives concentration for the number of copies $S\in\pzc{H}$ of $H$ that appear in $G(n,p)$.  Each copy is then kept with probability $q$, independently of the others, and so Chernoff's inequality (see e.g.\ \cite{randomgraphbook}*{Theorem 2.1}) gives concentration for the number of these copies that appear in $\binom{G(n,p)}{H}_q$.  
 
 \subsubsection{An exponential upper tail bound}

 Janson's inequality (Lemma \ref{lem:vanillajanson}) allows us to conclude that the probability that the number of embeddings of a graph $H$ in $G(n,p)$ is significantly smaller than its expectation is exponentially small. On the other hand, one can use Lemma \ref{lem:chebyshev} to give a bound on the probability that the number of copies of $H$ is much higher than expected. However, the concentration given by Lemma \ref{lem:chebyshev} is not enough for our purposes. We therefore need the following bound for the upper tail of the distribution of subgraph counts in random graphs. This is a simplification of the main result in \cite{janson2004upper} and the proof is almost identical. The only departing point from the exposition in \cite{janson2004upper} is that Lemma \ref{lem:upper tail lemma} allows us to apply the result to the number of copies of $H$ on prescribed vertex subsets\footnote{In fact, although we will only use this result to restrict to copies of $H$ that lie on prescribed vertex sets, note that the statement could be applied in more generality, to an arbitrary (large enough) family of potential copies.} (as opposed to the total number of copies of $H$ in $G(n,p)$). This results in a factor of $1/\epsilon^m$ in the upper bound of $\mathbb{E}[X^m]$ in the proof of  \cite{janson2004upper}*{Theorem 1.2}, which can be counteracted by choosing a smaller constant $c>0$ below.
 
 \begin{lem} \label{lem:upper tail lemma}
 Let $H$ be a  graph with at least one edge, $\epsilon>0$ and $p=p(n)$ such that $\Phi:=\Phi_{H,p}\geq 1$. Then there exists some $c=c(H,\epsilon)>0$ such that the following holds. Let $\pzc{H}\subset \binom{[n]}{H}$ be some family of $\epsilon n^{v_H}$ potential copies of $H$. Letting $X$ be the random variable that counts the number copies $S\in\pzc{H}$ which appear in $G(n,p)$ on $[n]$, we have that \[ \mathbb{P}[X\geq 2\mathbb{E}[X]]\leq  \exp \left(-c\Phi^{\frac{1}{e_H}}\right).  \]
 \end{lem}

\subsubsection{Kim--Vu polynomial concentration} The last tool we need is the result of Kim and Vu \cite{kim2000concentration} (see also \cite{alonspencer}*{Section 7.8}). We state here a simplified version which is catered to our purposes. 

\begin{lem}\label{lem:kimvu}
Given $k\in \mathbb{N}$, let $c:=8^{-1}(4k!)^{-1/(2k)}$, and let $\bm{H}=(V,E)$ be a $k$-uniform hypergraph with $|V|=N$, $|E|=M$. Now consider the set $V'$ obtained by keeping each vertex of $V$ with some probability $q=q(N)\in[0,1]$, independently of the other vertices. We are interested in the random variable $Y:=e_{\bm{H}[V']}$ and we fix  $\mu:=\mathbb{E}[Y]=Mq^k$. Then
setting \[\nu:=\max_{1\leq i\leq k}\max_{I\in\binom{[N]}{i}}\deg_{\bm{H}} (I)q^{k-i},\] we have that if $\nu\leq \mu$, then 
\[\mathbb{P}\left(|Y-\mu|\geq \frac{\mu}{2}\right)\leq 2e^2N^{k-1}\exp\left(-c\left(\frac{\mu}{\nu}\right)^{\frac{1}{2k}}\right).\]
\end{lem}

\subsection{Proof of Theorem \ref{thm:robustKreuter1}} \label{sec:robustKreuter1statement}
 Towards proving Theorem~\ref{thm:robustKreuter1}, we first prove some lemmas. For a fixed nonempty graph $H$, we define
\[\cH(H):=\{H_1\cup H_2: H_1\equiv H_2\equiv H, V(H_1)\cap V(H_2)\neq \emptyset, H_1\cup H_2 \not\equiv H\}\]
to be the set of graphs that can be obtained by taking  the union of two distinct copies of $H$  which intersect in at least one vertex.  Recall also the definition of $\binom{G(n,p)}{H}_q$ from Section~\ref{sec:janson}.

\begin{lem} \label{lem:overlapping copies}
Let $H$ be a  graph with at least one edge, and let $\Phi':=\min\{n,\Phi_{H,p}\}$. Then there exists $C=C(H)>0$ such that the following holds for all  $q=q(n,p)$ such that $q\leq \Phi'/(n^{v_H}p^{e_H})$. Let $X_H = \card{ \binom{G(n,p)}{H}_q}$ be the number of selected copies of $H$ in $G(n,p)$, and for each $\bar{H} \in \cH(H)$, let $Y_{\bar H}$ be the random variable counting the number of intersecting pairs of copies of $H$ in $\binom{G(n,p)}{H}_q$ whose union is isomorphic to $\bar H$. Then we have that 
\[\mathbb{P}\left(\sum_{\bar H\in \cH(H)}Y_{\bar H}\geq  \frac{C\mathbb{E}[X_H]^2}{\Phi'}\right)\leq \frac{C\Phi'}{\mathbb{E}[X_H]^2}. \]
\end{lem}
\begin{proof}
This is a simple application of Chebyshev's inequality, Lemma \ref{lem:chebyshev}. The proof of \cite{randomgraphbook}*{Theorem 3.29} contains a similar calculation. 

Let us fix some $\bar H=H_1\cup H_2\in \cH(H)$ and show that $Y_{\bar H}\geq C'\mathbb{E}[X_H]^2/\Phi'$ with probability at most $\bar C\Phi'/\mathbb{E}[X_H]^2$ for some $C', \bar C>0$. The conclusion will then follow by a union bound, as there are only finitely many possible $\bar H\in \cH(H)$. First, let us upper bound the expectation of $Y_{\bar H}$ as follows. Defining $J=H_1\cap H_2$ as the intersection of the two copies of $H$ that comprise $\bar H$, we have that 
\begin{align*}
    \mathbb{E}[Y_{\bar H}]&\leq q^2n^{v_{(H_1\cup H_2)}}p^{e_{(H_1\cup H_2)}}  
    = \frac{q^2n^{2v_H}p^{2e_H}}{n^{v_J}p^{e_J}} 
    \leq \frac{C'\mathbb{E}[X_H]^2}{2\Phi'},
\end{align*}
for some appropriately defined $C'> 0$, using that $v_J\geq 1$ and $n^{v_J}p^{e_J}\geq \Phi_{H,p}$ if  $e_J\neq 0$. We now turn to concentration and look to apply Lemma \ref{lem:chebyshev}. In order to do this, we need an upper bound estimate on $\Delta$, which counts the expected number of non-independent pairs of copies of $\bar H$, each of whose two copies of $H$ are in $\binom{G(n,p)}{H}_q$. In particular, it counts the number of pairs of copies of $\bar H$ which overlap in at least one edge. So let us fix some graph $H^*=H_1\cup H_2 \cup  H'_1\cup H_2'$ such that both $H_1\cup H_2$ and $H_1'\cup H'_2$ are  copies of $\bar H$,  each $H_i$ and $H_i'$ is copy of $H$, and $H_1\cup H_2$ intersects $H'_1\cup H_2'$ in at least an edge, and let $Y_{H^*}$ count the number of quadruples of copies of $H$ in $\binom{G(n,p)}{H}_q$ whose union is isomorphic to $H^*$. There are finitely many such $H^*$, and we have the upper bound $\Delta \le \sum_{H^*} \mathbb{E}[Y_{H^*}]$.

Given such a configuration $H^*$, let $\tilde{x}$ be $2$ if $H^*=H_1\cup H_2=H_1'\cup H_2'$ is a single copy of $\bar H$, $1$ if $H_i=H'_j$ in $H^*$ for some $i,j\in \{1,2\}$ and $0$ otherwise. In other words, $\tilde{x}$ indicates the number of `repeated' copies of $H$ in $H^*$. Swapping the indices $1$ and $2$ if necessary, let $J_1=H_1\cap H_2$, $J_2=H_1'\cap(H_1\cup H_2)$ and $J_3=H_2'\cap (H_1\cup H_2 \cup H_1')$ such that each $J_i$ contains at least one vertex. This is possible due to the fact that $H_1$ and $H_2$ intersect in at least a vertex in $\bar H$ and the two copies of $\bar H$ intersect in at least an edge. Then we have that 
\begin{align*}
    \mathbb{E}[Y_{H^*}]&\leq q^{4-\tilde{x}}n^{v_{(H_1\cup H_2\cup H'_1 \cup H'_2)}}p^{e_{(H_1\cup H_2\cup H'_1 \cup H'_2)}} \\
    & = \frac{q^{4-\tilde{x}}n^{4v_H}p^{4e_H}}{n^{v_{J_1}+v_{J_2}+v_{J_3}}p^{e_{J_1}+e_{J_2}+e_{J_3}}}  \\
&\leq C'' \mathbb{E}[X_H]^{4-\tilde{x}}/(\Phi')^{3-\tilde{x}} \\
&\leq C''\mathbb{E}[X_H]^2/\Phi',
\end{align*}
for appropriately defined $C''>0$, using that  at least $\tilde{x}$ of  $J_2$ and 
$J_3$ are copies of $H$, and using that $\mathbb{E}[X_H]\leq qn^{v_H}p^{e_H}\leq \Phi'$ in
the final step (recall that by hypothesis $q \leq \Phi ' /(n^{v_H}p^{e_H})$).
Thus, summing over all possible $H^*$, we get that $\Delta \leq \tilde{C}
\mathbb{E}[X_H]^2/\Phi'$ for some $\tilde{C}>0$ and by Lemma \ref{lem:chebyshev},
\begin{align*}
    \mathbb{P}\left(Y_{\bar H}\geq \frac{C' \mathbb{E}[X_H]^2}{\Phi'}\right)&\leq \mathbb{P}\left(Y_{\bar H} \geq \mathbb{E}[X_{\bar H}]+\frac{C'\mathbb{E}[X_H]^2}{2\Phi'}\right) 
     \leq \frac{\bar C \Phi'}{\mathbb{E}[X_H]^2},
\end{align*}
for $\bar C=4\tilde{C}/C'^2$.  Thus summing over all $\bar{H}\in \cH(H)$ and taking a union bound on the failure probabilities, we can choose $C>0$ appropriately so that the statement of the lemma is satisfied. 
\end{proof}
 
 Let $\bigsqcup^kH$ denote the graph obtained by taking $k$ vertex-disjoint copies of $H$.
 We say that a set $K\in\binom{[n]}{k}$ is a \emph{transversal} of a copy $S$ of
 $\bigsqcup^kH$ if $K$ contains one vertex from each of the copies of $H$ that comprise $S$.
 Further, given $\cK\subseteq \binom{[n]}{k}$, we say a copy of $\bigsqcup^kH$ on $[n]$ is
 \emph{$\cK$-spanning} if it contains a set from $\cK$ as a transversal. 
 
 \begin{lem} \label{lem:dangersets}
 Let $H$ be a  graph with at least one edge, $k\in \mathbb{N}$ and $\delta>0$. 
 Set $\delta_1:= \delta (kv_H)^k$.
 Then there exists a $c>0$ and $n_0\in \mathbb{N}$ such that if $p=p(n)\geq n^{-1/m(H)}$ and $q=q(n,p)$ satisfy $qn^{v_H}p^{e_H}>(\log n)^{3k}$ then the following holds for all $n\geq n_0$.  Suppose that $\cK\subset \binom{[n]}{k}$ is such that $|\cK|\leq \delta n^k$. Letting $Y$ be the random variable that counts the number of $\cK$-spanning copies of $\bigsqcup^kH$, composed of copies of $H$ in $\binom{G(n,p)}{H}_q$,
 we have that 
 \[\mathbb{P}(Y\geq 4 \delta_1 ( n^{v_H}p^{e_H}q)^k)\leq \exp\left(-c \Phi(H,p)^{\frac{1}{ke_H}}\right)+\exp\left(-c\left(qn^{v_H}p^{e_H}\right)^{\frac{1}{2k}}\right).\]
 \end{lem}
 \begin{proof}
 It suffices to prove the lemma in the case when $|\cK|=\delta n^k$.
 For this, we split the analysis of $\binom{G(n,p)}{H}_q$ into looking at the random edges given by $G(n,p)$ and the random function $\rho(H;q):\binom{[n]}{H}\rightarrow\{0,1\}$ separately. Firstly consider the $\cK$-spanning copies of $\bigsqcup^kH$ in the complete graph $K_n$. There are at most $\delta _1 n^{kv_H}$ such copies  and each appears with probability $p^{ke_H}$. Thus, in expectation, the number of $\cK$-spanning copies of $\bigsqcup^kH$ in $G(n,p)$ is at most $\delta _1(n^{v_H}p^{e_H})^k$. Moreover, Lemma~\ref{lem:upper tail lemma} tells us that the count of such copies in $G(n,p)$ is at most twice this with probability at least $1-\exp\left(-c_{\ref{lem:upper tail lemma}}\Phi^{\frac{1}{ke_H}}\right),$ where $c_{\ref{lem:upper tail lemma}}=c(H,\delta)$ as given by Lemma~\ref{lem:upper tail lemma} and $\Phi=\Phi\left(\bigsqcup^kH,p\right)$. Now note that 
 \begin{align*}
     \Phi\left(\bigsqcup\nolimits^kH,p\right)&=\min\left\{\prod_{i\in[k]}n^{v_{J_i}}p^{e_{J_i}}:J=\bigsqcup_{i\in[k]}J_i\subseteq \bigsqcup\nolimits^kH, e_J>0 \right\} \\
     &\geq \min\left\{n^{v_{J_j}}p^{e_{J_j}}:J=\bigsqcup_{i\in[k]}J_i\subseteq \bigsqcup\nolimits^kH, e_{J_j}>0\right\}
     \\ &\geq \Phi(H,p),
 \end{align*}
 where we split subgraphs $J\subseteq \bigsqcup^kH$ according to their subgraphs $J_i$ in the $i^{th}$ copy of $H$ in $\bigsqcup^kH$ and in the second step we single out a $j=j(J)$ such that $J_j\subseteq J$ has a nonempty edge set. 
 
 Applying Lemma \ref{lem:upper tail lemma} also to the counts of $\bigsqcup^{k'}H$, for smaller values of $k'$, we can conclude that there exists a $c'>0$ so that,
 with probability at least 
 \[1-\exp\left(-c'\Phi(H,p)^{\frac{1}{ke_H}}\right),\] 
 there are at most $2\delta_1(n^{v_H}p^{e_H})^k$ $\cK$-spanning copies of $\bigsqcup^kH$ in $G(n,p)$ and there are at most $2(n^{v_H}p^{e_H})^{k'}$ copies of $\bigsqcup^{k'}H$ in $G(n,p)$ for all $1\leq k'\leq k-1$. 
 On the other hand there are at least $\delta n^{kv_H}/(2kv_H)!$ $\cK$-spanning copies of $\bigsqcup^kH$
 in $K_n$. So by Lemma~\ref{lem:vanillajanson} there is a $c''>0$ such that with probability at least
 \[1-\exp\left(-c''\Phi(H,p)\right),\] 
 there are at least $2\delta _2 (n^{v_H}p^{e_H})^k$ $\cK$-spanning copies of $\bigsqcup^kH$
 in $G(n,p)$ where $\delta _2:= \delta /4(2kv_H)!.$

 Now we condition on all these events occurring in $G(n,p)$ and turn to analyse the effect of $\rho(H,q)$. We know that the probability that each copy of $H$ in a $\cK$-spanning copy of $\bigsqcup^k H$ is selected in $\binom{G(n,p)}{H}_q$ with probability $q^k$, and we will obtain concentration via a simple application of Lemma~\ref{lem:kimvu}. Indeed, consider the auxiliary $k$-uniform hypergraph $\bm{H}$ whose vertex set is given by copies of $H$ in $G(n,p)$ and whose edge set is given by copies of $H$ which comprise a $\cK$-spanning copy of $\bigsqcup^kH$.  From above we have that $\bm{H}$ has at most $2n^{v_H}p^{e_H}$ vertices (the copies of $H$ in $G(n,p)$) and 
 between $2\delta_2(n^{v_H}p^{e_H})^{k}$ and
 $2\delta_1(n^{v_H}p^{e_H})^{k}$ edges. We also know, from the concentration on the number of copies of $\bigsqcup^{k'}H$ in $G(n,p)$, that for any set $I$ of $i$ copies of $H$ with $1\leq i \leq k$, the number of edges of $\bm{H}$ containing $I$ is at most $2(n^{v_H}p^{e_H})^{k-i}$. Thus, Lemma \ref{lem:kimvu} tells us that conditioning on the outcome of $G(n,p)$ as above, with probability at least \[1-2e^2\left(2n^{v_H}p^{e_H}\right)^{k-1}\exp\left(-c_{\ref{lem:kimvu}}\left(\delta _2 n^{v_H}p^{e_H}q\right)^{\frac{1}{2k}}\right),\]
 the number of $\cK$-spanning copies of $\bigsqcup^kH$ whose copies of $H$ are all selected in $\binom{G(n,p)}{H}_q$ is at most $4\delta_1(n^{v_H}p^{e_H}q)^k$, where $c_{\ref{lem:kimvu}}$ is the constant given by Lemma \ref{lem:kimvu}. The conclusion 
 then follows from a simple calculation on the error probability that either the counts in $G(n,p)$ are not as desired or the count in $\binom{G(n,p)}{H}_q$ is too high, given that we get the desired counts in $G(n,p)$.
 \end{proof}
 
We now turn to proving Theorem \ref{thm:robustKreuter1}.

\begin{proof}[Proof of Theorem \ref{thm:robustKreuter1}]
It suffices to prove the theorem in the case when $p=Cn^{-1/m_K (F,H)}$ for some
sufficiently large $C>0$. 
  We begin with a calculation. Let 

 \begin{equation} \label{eq:ell} \ell:= \min_{J\subseteq H,e_J>0} \left(v_J-\frac{e_J}{m_K(F,H)}\right),  \end{equation} so that $Cn^\ell\leq\Phi_{H,p}\leq C^{e_H}n^{\ell}$. 
Letting $c_1:=m_1(F)/m_K(F,H)>0$, we have that $c_1 \leq \ell \leq 1$. Indeed let $J$ be the minimising subgraph of $H$ in the definition of $\ell$ and let $J'$ be the maximising subgraph of $H$ in the definition of $m_K(F,H)$. The lower bound on $\ell$ then follows from the fact that 
\begin{equation}\label{eq:densities}v_Jm_K(F,H)-e_J\geq v_J\left(\frac{m_1(F)+e_J}{v_J}\right)-e_J=m_1(F). \end{equation}
The upper bound on $\ell$ follows because 
\begin{align}
    m_1(F)\leq m_K(F,H)=\frac{m_1(F) + e_{J'}}{v_{J'}} 
    &\Longrightarrow \frac{m_1(F)}{m_1(F)+e_{J'}}\leq \frac{1}{v_{J'}}, 
\end{align}
and hence
\begin{align*}
    \ell &\leq v_{J'}-\frac{e_{J'}}{m_K(F,H)} 
     = v_{J'}\left(1-\frac{e_{J'}}{m_1(F)+e_{J'}}\right) 
    =v_{J'}\left(\frac{m_1(F)}{m_1(F)+e_{J'}}\right) 
    \leq 1.
\end{align*}

Now we turn to the proof of the theorem. We first fix constants.  We fix $\delta_0>0$ such
that 
\begin{align}\label{newnewnew}
  \delta_0<\frac{1}{32\cdot v_F! (4(2v_H)! v_Fv_H)^{v_F} } 
\end{align}
and $c>0$ such that $c<\frac{c_1}{4v_Fe_Fe_H}$.
Further, for each $0\leq i\leq t$ we fix $\eta_i:=|U_i|/n$ so that $\eta_i\geq \eta_0$ for
all $i$. Further, fix 
\[\gamma_i:=\frac{(1/(v_H!)-\delta)^2(\eta _0\eta_i)^{v_H}}{16C_{\ref{lem:overlapping copies}}}\]
for all $0\leq i\leq t$, where  $C_{\ref{lem:overlapping copies}}=C_{\ref{lem:overlapping copies}}(H)$ is the constant obtained from Lemma \ref{lem:overlapping copies}. 
By considering a large enough constant $C$, we expose $G(n,p)$ in two rounds so that $G(n,p)=G_1(n,p_1)\cup G_2(n,p_2)$, with $p_1,p_2\geq C'n^{-1/m_K(F,H)}$ for $C'$  such that $C'>\frac{2\log v_H}{\gamma_0^{v_F}}$.
 Let us briefly sketch the proof which splits into proving two main claims. The first claim states that with high probability in $G_1$, for each $i\in [t]$, there is a (large) subfamily $\pzc{D}_i\subset\pzc{H}_i$ of pairwise vertex-disjoint copies of $H$, all of whose edges appear in $G_1$. We define 
\begin{equation} \label{eq:transversals} \cW_i:= \{W\subset U_i: |W\cap T|=1 \mbox{ for all } T\in \pzc{D}_i\} \end{equation} to be the sets which can be obtained by choosing one vertex from each copy of $H$ in $\pzc{D}_i$.  The second claim is that with high probability in $G_2$, for each $i$ and each set $W\in \cW_i$, there is a copy of $F$ which lies in $\binom{W}{F}\cap \pzc{F}_i$ whose edges appear in $G_2$.  The proof then follows easily from these two claims. Indeed, consider a red/blue colouring of $G(n,p)=G_1\cup G_2$ and some $i\in [t]$. If there is no blue copy of $H$ in $\pzc{H}_i$ then in particular, every copy of $H$ in $\pzc{D}_i$ must contain a red vertex. By choosing one red vertex in each copy $T$ of $H$ in $\pzc{D}_i$, we get a set $W\in \cW_i$ which is entirely red. The second claim then tells us that this set hosts a copy of $F$ which lies in $\pzc{F}_i$ and so we are done. It remains to prove the two claims above.

\smallskip

In order to prove the first claim, it will be useful to consider only the selected copies $\binom{G_1(n,p_1)}{H}_q$ as introduced in Section~\ref{sec:janson}. As these copies all appear in $G_1(n,p_1)$, it will suffice to find a suitable family $\pzc{D}_i \subseteq \binom{G_1(n,p_1)}{H}_q$. So we fix \[
q:=\frac{\eta _0^{v_H}(1/(v_H!)-\delta)}{4C_{\ref{lem:overlapping copies}}}\left(\frac{n^\ell}{n^{v_H}{p_1}^{e_H}}\right)=\sqrt{\frac{\gamma_0}{C_{\ref{lem:overlapping copies}}}}\left(\frac{n^\ell}{n^{v_H}{p_1}^{e_H}}\right).\]  
Now we apply Lemma \ref{lem:overlapping copies}, observing that $\Phi'\geq n^\ell$ due to our calculation at the beginning of this proof ($\Phi'=\Phi_{H,p_1}\geq n^\ell$ if $\ell<1$ and $\Phi'=n^\ell=n$ if $\ell=1$). 
As the expected number of copies of $H$ in $\binom{G_1(n,p_1)}{H}_q$ is $\Omega (qn^{v_H}{p_1}^{e_H}) =\Omega(n^\ell)$, we have that with high probability (with probability at least $1-O(n^{-\ell})$) there are at most \[\frac{C_{\ref{lem:overlapping copies}}q^2n^{2v_H}{p_1}^{2e_H}}{n^\ell}=\gamma_0n^\ell\] overlapping copies of $H$ in $\binom{G_1(n,p_1)}{H}_q$.  For a given $i\in [t]$, we can conclude from Lemma \ref{lem:janson} that with high probability there at least 
\[\frac{(1/(v_H!)-\delta)q|U_i|^{v_H}{p_1}^{e_H}}{2}\geq 2\gamma_in^\ell\]
copies of $H$ in $\pzc{H}_i \cap \binom{G_1(n,p_1)}{H}_q$. As this holds with probability at least $1-\exp(-n^{c_1})$, we have that this holds for all $i \in [t]$ with high probability. Thus we obtain a family $\pzc{D}_i\subset \pzc{H}_i$ of  vertex-disjoint copies of $H$ by taking the copies in $\pzc{H}_i \cap \binom{G_1(n,p_1)}{H}_q$  and deleting one copy from any pair of overlapping copies. Our calculations above guarantee that with high probability, for every $i\in [t]$, $\pzc{D}_i$ has  size at least 
$\tilde{n}_i:=\gamma_in^\ell$, and we restrict each family to one of size exactly $\tilde{n}_i$.  

\smallskip

We now turn to the second exposure, namely $G_2=G_2(n,p_2)$, and look to prove that for every $i\in [t]$ and every set $W\in \cW_i$ there is a copy of $F$ in $\binom{W}{F}\cap \pzc{F}_i$ which appears in $G_2$, where $\cW_i$ is as defined in (\ref{eq:transversals}). Fixing an $i\in[t]$ and a $W\in \cW_i$, we consider $G_2$ restricted to $W$. We look to apply Lemma  \ref{lem:vanillajanson} and so need a lower bound on the parameter \[\tilde{\Phi}_i:=\Phi_{F,p_2}=\min_{I\subseteq F, e_I>0}\tilde{n}_i^{v_I}p_2^{e_I},\] which is calculated with respect to the vertex set $W$. As at the beginning of the proof, we  set $J\subset H$ to be the minimising subgraph in the definition of $\ell$ (\ref{eq:ell}) and  use that for all $I\subseteq F$ with $e_I>0$, we have that 
\[\ell m_K(F,H)=v_Jm_K(F,H)-e_J\geq m_1(F)\geq \frac{e_I}{v_I-1},\]
by (\ref{eq:densities}). Rearranging, we obtain that 
\[v_I-\frac{e_I}{\ell m_{K(F,H)}}\geq 1,\]
and so 
\[\min_{I\subseteq F, e_I>0} \left( \ell v_I-\frac{e_I}{m_K(F,H)} \right) \geq \ell. \]
We conclude that $\tilde{\Phi}_i\geq \tilde{\Phi}_0 \geq \gamma_0^{v_F}C'n^{\ell}$. As  $t \leq \exp(n^c)$ and for each $i$, and $|\cW_i|= (v_H)^{\tilde{n}_i} \leq \exp(n^\ell\log v_H )$, we can take a union bound and conclude from Lemma \ref{lem:vanillajanson} that for all choices  of $i\in[t]$ and $W\in \cW_i$, we have that there are at least $\frac{\tilde{n}_i^{v_F}{p_2}^{e_F}}{2}$ copies of $F$ on $W$ in $G_2$ with high probability. Note here that we used that $C'>\frac{2\log v_H}{\gamma_0^{v_F}}$. It remains to prove that for each $i\in[t]$ and $W$ one of these copies of $F$ belongs to $\pzc{F}_i$. 

To this end we define $\pzc{B}_i:=\binom{U_i}{F}\setminus \pzc{F}_i$ to be the copies of $F$ which do not lie in our desired collection. We will upper bound the number of copies $S$ of $F$ in $\pzc{B}_i$ which appear in $G_2$, such that each vertex of $S$ lies in a different copy of $H$ in $\pzc{D}_i$. In order to do this, we return to analyse our construction of $\pzc{D}_i$ and in particular our use of $\binom{G_1(n,p_1)}{H}_q$. Let $\cK_i$ be the collection of $v_F$-sets in $U_i$ which host a copy of $F$ in $\pzc{B_i}$ and note that $|\cK_i|\leq \delta |U_i|^{v_F}$. We say a set $K\in \cK_i$ is \emph{dangerous} if each vertex of $K$ is contained in a distinct copy of $H$ in $\pzc{D}_i$. In order to be dangerous, a set $K$ has to lie in a transversal of a copy of $\bigsqcup^{v_F}H$, composed of copies of $H$ in $\binom{G_1(n,p_1)[U_i]}{H}_q$ (see the paragraph before Lemma~\ref{lem:dangersets} for the relevant definitions). Therefore in order to upper bound the number of dangerous sets, it suffices to upper bound the number of $\cK_i$-spanning copies of $\bigsqcup^{v_F}H$ in $\binom{G_1(n,p_1)[U_i]}{H}_q$. It follows then from Lemma \ref{lem:dangersets} that for all $i\in [t]$, there are at most 
\[4\delta (v_Fv_H)^{v_F} \left(|U_i|^{v_H}{p_1}^{e_H}q\right)^{v_F}=
\frac{(v_Fv_H)^{v_F} \cdot2^{2v_F+2}\delta}{(1/(v_H!)-\delta)^{v_F}} \tilde{n}_i^{v_F}\stackrel{(\ref{newnewnew})}{\leq }\frac{\tilde{n}_i^{v_F}}{8v_F!} \]
dangerous sets with high probability, using that $\delta<\delta_0$. As for a fixed $i\in[t]$, this holds with probability $1-\exp\left(-\Omega\left(n^{\frac{\ell}{2v_Fe_H}}\right)\right)$ and $t\leq \exp(n^c)$, we can conclude that there at most $\tilde{n}_i^{v_F}/(8v_F!)$ dangerous sets for each $i\in[t]$ with high probability.

Finally, we calculate how many copies of $F$ in $G_2$ are hosted on dangerous sets. For a fixed $i$, we consider $G_2$ restricted to the vertex set $D_i:=\cup_{T\in \pzc{D}_i}V(T)$. We have that $|D_i|=v_H\tilde{n}_i$ and from the previous paragraph we may assume that there are at most $\tilde{n}_i^{v_F}/8$ potential copies of $F$ on dangerous sets in $D_i$. 
Each of these appears with probability $p_2^{e_F}$ and by Lemma \ref{lem:upper tail lemma}  we have that with probability at least $1-\exp\left(-\Omega\left(n^{\frac{\ell}{e_F}}\right)\right)$, there are at most $\tilde{n}_i^{v_F}p_2^{e_F}/4$ copies of $F$ in $G_2$ which are hosted on dangerous sets. 
The failure probability here follows from a calculation of the appropriate $\Phi_{F,p_2}$ similar to the calculation of $\tilde{\Phi}$ above. 
Thus we can take a union bound to conclude that for all $i\in[t]$, there are at most $\tilde{n}_i^{v_F}p_2^{e_F}/4$ copies of $F\in \cB_i$ which lie in $D_i$, whose edges appear in  $G_2$ and whose vertices are contained in distinct copies of $H$ in $\pzc{D}_i$. Thus with high probability, for all $i\in[t]$ and for all $W\in \cW_i$, there is a copy  $T\in\binom{W}{F}\cap \pzc{F}_i$ of $F$, whose edges appear in $G_2$, as required. 
\end{proof}

\section{Concluding remarks}\label{sec:conc}

In this paper, we have determined, at essentially every density $d$, the perturbed vertex Ramsey threshold $p(n; K_r, H, d)$ for cliques versus arbitrary graphs.  One could investigate how these thresholds change with the introduction of additional colours, but the most pressing problem that remains open is to extend our results to all pairs of graphs $(F,H)$, with the symmetric case $F = H$ of particular interest.  Our methods do provide lower and upper bounds on the threshold in the general case, which we discuss below.

\smallskip

We start with the $1$-statement, where we wish to know what $p$ ensures $G_n \cup G(n,p)$ is $(F,H)_v$-Ramsey when $G_n$ is a graph of density more than $1 - 1/(k-1)$.  Recall that in our algorithmic proof of the $1$-statement in Theorem~\ref{thm:main}, we worked in an $\eps$-regular $k$-tuple in $G_n$, using the vertex Ramsey properties of the random graph in each part to iteratively grow a red clique or try to build a copy of $H$.

In the general setting, when we seek a red copy of $F$ instead, we can adopt the same approach.  The main difference is that there are many ways we could try to build $F$ over the $k$ parts.  To keep track of these, we define a \emph{partial partition of $F$} to be a partition of the vertices $V(F) = U_1 \cup \hdots \cup U_k \cup W$, where $W \neq \emptyset$.  This represents the stage in the algorithm where we have found red subgraphs $F[U_i]$ in the parts $V_i$, and $W$ represents the vertices of $F$ that are still missing.  Thus, when we try to extend this red subgraph, we will require $G(n,p)[V_i]$ to be $(F[U_i \cup \{u_i\}], H_i)_v$-Ramsey for some optimal choice of $u_i \in W$ and partition $H = H_1 \cup \hdots \cup H_k$.  In this way, we either get one vertex closer to having a red copy of $F$, or we find one of the parts we need for a blue copy of $H$.  Our proof then shows that, if we write $p(n; F, H, d) =: n^{-1/m^*(F, H; k)}$, we have
\begin{equation} \label{eqn:gen1stata}
m^*(F, H; k) \le \max\limits_{\substack{V(F) = U_1 \cup \hdots \cup U_k \cup W;\\W \neq \emptyset}} \min\limits_{\substack{H = H_1 \cup \hdots \cup H_k; \\ u_1, \hdots, u_k \in W}} \max\limits_{i : H_i \neq \emptyset} \beta(F[U_i \cup \{u_i\}], H_i).
\end{equation}
Note that when $F = K_r$, all that matters is the size $\card{U_i}$ and not the set $U_i$ itself, since each induced subgraph of $K_r$ is itself a clique.  Hence we recover the bound of Theorem~\ref{thm:main}.

Unfortunately, this upper bound need not be tight.  For instance, when $F$ and $H$ are complete bipartite graphs, it is not hard to see that $m^*(F, H; 4) = 0$.  However, by considering sets $U_i$ that, for each $i$, span both colour classes of $F$, we can ensure that each subgraph $F[U_i]$ has edges, which results in the right-hand side of~\eqref{eqn:gen1stata} being positive.

One issue with~\eqref{eqn:gen1stata} is that it considers \emph{all} partial partitions of $F$, but we need only maximise over those that could feasibly arise in the algorithm.  While it may not be easy to describe these partitions explicitly, we can construct the family of feasible partial partitions recursively.

To do so formally, we define the \emph{extension function} $f$ which, given the sets $U_1, \hdots, U_k$ of a partial partition of $F$, returns the vertices $(u_1, \hdots, u_k) \in W^k$ that are used to extend the red subgraph.  For each such extension function, we can build the family $\mc F(f)$ of feasible partitions in the following way.  We start with $(\emptyset, \hdots, \emptyset) \in \mc F(f)$.  Then, for each $i \in [k]$, we add $(U_1, \hdots, U_{i-1}, U_i \cup \{f(\vec{U})_i\}, U_{i+1}, \hdots, U_k)$ to $\mc F(f)$, provided this is still a partial (and not complete) partition of $F$.  Note that this represents the larger red subgraph we would obtain if, in $G(n,p)[V_i]$, we would find a monochromatic red subgraph when applying the $(F[U_i \cup \{f(\vec{U})_i \}], H_i)_v$-Ramsey property.

We then need only maximise over the feasible partitions $\mc F(f)$, and can choose the extension function $f$ that gives the lowest possible threshold.  That is, we have the upper bound
\begin{equation} \label{eqn:gen1statb}
m^*(F,H;k) \le \min\limits_{f} \max\limits_{(U_1, \hdots, U_k) \in \mc{F}(f);} \min\limits_{H = H_1 \cup \hdots \cup H_k;} \max\limits_{i: H_i \neq \emptyset} \beta(F[U_i \cup \{f(\vec{U})_i\}], H_i).
\end{equation}
Note again that in the case $F = K_r$, the choice of extension function $f$ is irrelevant, since all that matters are the sizes $\card{U_i}$.

This is strictly better than~\eqref{eqn:gen1stata}, as one can find an extension function $f$ that shows $m^*(F,H;4) = 0$ whenever $F$ and $H$ are bipartite.  Unfortunately, even~\eqref{eqn:gen1statb} need not be tight, as we should also have $m^*(F,H; 3) = 0$ for such $F$ and $H$, but the right-hand size is positive when we only have three parts. It would therefore be very interesting to find a sharper bound for the $1$-statement.  A useful step in that direction could be to characterise which extension functions $f$ are optimal for a given graph $F$.

\smallskip

In the other direction, we can provide lower bounds on $m^*(F, H;k)$ by generalising the colouring we gave in proving the $0$-statement of Theorem~\ref{thm:main}. We shall once again take $G_n$ to be a complete $k$-partite graph, and will describe how one can colour the vertices of the random graphs $G(n,p)[V_i]$ to avoid both a red $F$ and a blue $H$ in $G_n \cup G(n,p)$.

To this end, we call a $k$-tuple $(\mc F_1, \hdots, \mc F_k)$ of families of nonempty graphs a \emph{$k$-cover} of $F$ if, for any $k$-partition $V(F) = U_1 \cup \hdots \cup U_k$ of the vertices of $F$, there is some $i \in [k]$ such that $F[U_i] \in \mc F_i$. That is, a $k$-cover is a collection of induced subgraphs that are bound to appear in any $k$-partition of $F$.

Given this definition, we have the following lower bound.
\begin{equation} \label{eqn:gen0stat}
m^*(F,H;k) \ge \max\limits_{(\mc F_1, \hdots, \mc F_k) \text{ $k$-cover of }F;} \min\limits_{H = H_1 \cup \hdots \cup H_k;} \max\limits_{i: H_i \neq \emptyset;} \min\limits_{F' \in \mc F_i} \beta(F',H_i).
\end{equation}
When $F = K_r$, this recovers the bound from Theorem~\ref{thm:main}, since we have $k$-covers  of the form $\mc F_i = \{ K_{r_i + 1} , \dots, K_r\}$, where $\sum_i r_i = r-1$.

To describe the colouring in the general case, fix a maximising $k$-cover $(\mc F_1, \hdots, \mc F_k)$, let $\beta^*$ be the right-hand side of~\eqref{eqn:gen0stat}, and let $p = o(n^{-1/\beta^*})$. For each $i \in [k]$, we define $\mc H_i = \{H' \subseteq H : \forall F' \in \mc F_i, \, \beta(F', H') \ge \beta^* \}$. As before, one can argue that $H \in \mc H_i$, and so these families are all nonempty. Applying Proposition~\ref{prop:0statementforfams}, we can colour the vertices of $G(n,p)[V_i]$ so as to avoid any red graph from $\mc F_i$ and any blue graph from $\mc H_i$.

It is now tautological that this colouring of $G_n \cup G(n,p)$ has neither a red $F$ nor a blue $H$. Suppose for contradiction there is a red copy of $F$, partitioned as $F = F_1 \cup \hdots \cup F_k$. Since $(\mc F_1, \hdots, \mc F_k)$ is a $k$-cover of $F$, there is some $i$ with $F_i \in \mc F_i$, but then there is no red $F_i$ in $G(n,p)[V_i]$. On the other hand, if there is a blue $H$, partitioned as $H = H_1 \cup \hdots \cup H_k$, then we must have some $i \in [k]$ such that $H_i \neq \emptyset$ and $\beta(F', H_i) \ge \beta^*$ for all $F' \in \mc F_i$. But then $H_i \in \mc H_i$, and so there is no blue $H_i$ in $G(n,p)[V_i]$ either.

The challenge arises from the fact that when $F$ is not a clique, there could be many ways to partition it into $k$ induced subgraphs, and so there will be a wide variety of complicated $k$-covers. This makes it hard to analyse~\eqref{eqn:gen0stat}, and in particular to compare it to the upper bound of~\eqref{eqn:gen1statb}. Indeed, it is not obvious at first sight that the right-hand side of~\eqref{eqn:gen1statb} is at least that of~\eqref{eqn:gen0stat}. To close the gap between the bounds, it would help to better understand the $k$-covers of a graph $F$, and to see if there are different colourings of $G_n \cup G(n,p)$ that show~\eqref{eqn:gen0stat} is not tight.
\section{Acknowledgement}
The authors are grateful to the referees for their careful and helpful reviews.

\bibliography{Biblio}

@BOOK{
    alonspencer,
	author = {N. Alon and J. H. Spencer},
	title = {The {P}robabilistic {M}ethod},
	publisher={Wiley},
	series = {Wiley series in Discrete Mathematics and Optimization},
	year = {2015},
    edition = {4th}
}

@BOOK{
    randomgraphbook,
	author = {S. Janson and T. \L{}uczak and A. Ruci\'nski},
	title = {Random Graphs},
	publisher={John Wiley and Sons},
	series = {Wiley-Interscience series in Discrete Mathematics and Optimization},
	year = {2000}
}

@article{luczak1992ramsey,
  title={Ramsey properties of random graphs},
  author={\L{uczak}, Tomasz and Ruci{\'n}ski, Andrzej and Voigt, Bernd},
  journal={Journal of Combinatorial Theory, Series B},
  volume={56},
  number={1},
  pages={55--68},
  year={1992},
  publisher={Academic Press}
}

@article{kreuter1996threshold,
  title={Threshold functions for asymmetric {R}amsey properties with respect to vertex colorings},
  author={Kreuter, Bernd},
  journal={Random Structures \& Algorithms},
  volume={9},
  number={3},
  pages={335--348},
  year={1996},
  publisher={Wiley Online Library}
}

@article{janson2004upper,
  title={Upper tails for subgraph counts in random graphs},
  author={Janson, Svante and Oleszkiewicz, Krzysztof and Ruci{\'n}ski, Andrzej},
  journal={Israel Journal of Mathematics},
  volume={142},
  number={1},
  pages={61--92},
  year={2004},
  publisher={Springer}
}

@article{janson1990poisson,
  title={Poisson approximation for large deviations},
  author={Janson, Svante},
  journal={Random Structures \& Algorithms},
  volume={1},
  number={2},
  pages={221--229},
  year={1990},
  publisher={Wiley Online Library}
}

@article{kim2000concentration,
  title={Concentration of multivariate polynomials and its applications},
  author={Kim, Jeong Han and Vu, Van H},
  journal={Combinatorica},
  volume={20},
  number={3},
  pages={417--434},
  year={2000},
  publisher={Springer}
}

@article{kohayakawa1997threshold,
  title={Threshold functions for asymmetric {R}amsey properties involving cycles},
  author={Kohayakawa, Yoshiharu and Kreuter, Bernd},
  journal={Random Structures \& Algorithms},
  volume={11},
  number={3},
  pages={245--276},
  year={1997},
  publisher={Wiley Online Library}
}

@article{kst,
  title={On smoothed analysis in dense graphs and formulas},
  author={Krivelevich, Michael and Sudakov, Benny and Tetali, Prasad},
  journal={Random Structures \& Algorithms},
  volume={29},
  pages={180--193},
  year={2006},
  publisher={Wiley Online Library}
}

@article{power,
  title={Ramsey properties of randomly perturbed dense graphs},
  author={Powierski, Emil},
  journal={arXiv:1902.02197},
  year={2019},
}

@article{mns,
  title={Towards the {K}ohayakawa--{K}reuter conjecture on asymmetric {R}amsey properties},
  author={Mousset, Frank and Nenadov, Rajko and Samotij, Wojitech},
  journal={Combinatorics, Probability and Computing},
  year={to appear},
}

@article{dastreg,
  title={Ramsey properties of randomly perturbed graphs: cliques and cycles},
  author={Das, Shagnik and Treglown, Andrew},
  journal={Combinatorics, Probability and Computing},
  year={to appear},
}

@article{random3,
  title={Threshold functions for {R}amsey properties},
  author={R\"odl, Vojtech and Ruci\'nski, Andrzej},
  journal={Journal of the American Mathematical Society},
  volume={8},
  pages={917--942},
  year={1995},
}

@article{bwt2,
  title={Tilings in randomly perturbed dense graphs},
  author={Balogh, Jozsef and Treglown, Andrew and Wagner, Adam},
  journal={Combinatorics, Probability and Computing},
  volume={28},
  pages={159--176},
  year={2019},
}

@article{bfkm,
  title={Adding random edges to dense graphs},
  author={Bohman, Tom and Frieze, Alan and Krivelevich, Michael and Martin, Ryan},
  journal={Random Structures \& Algorithms},
  volume={24},
  pages={105--117},
  year={2004},
}

@article{bfm1,
  title={How many edges make a dense graph {H}amiltonian?},
  author={Bohman, Tom and Frieze, Alan  and Martin, Ryan},
  journal={Random Structures \& Algorithms},
  volume={22},
  pages={33--42},
  year={2003},
}

@article{bhkmpp,
  title={Universality for bounded degree spanning trees in randomly perturbed graphs},
  author={ B\"ottcher, Julia and Han, Jie and Kohayakawa, Yoshi and
  Montgomery, Richard and  Parczyk, Olaf and Person, Yury},
  journal={Random Structures \& Algorithms},
  volume={55},
  pages={854--864},
  year={2019},
}

@article{bmpp2,
  title={Embedding spanning bounded degree graphs in randomly perturbed graphs},
  author={ B\"ottcher, Julia  and
  Montgomery, Richard and  Parczyk, Olaf and Person, Yury},
  journal={Mathematika},
   volume={66},
  pages={422--447},
  year={2020},
}

@article{kks1,
  title={Bounded-degree spanning trees in randomly perturbed graphs},
  author={Krivelevich, Michael and Kwan, Matthew  and Sudakov, Benny},
  journal={{SIAM} Journal on Discrete Mathematics},
  volume={31},
  pages={155--171},
  year={2017},
}

@article{dudek,
  title={Powers of {H}amiltonian cycles in randomly augmented graphs},
  author={Dudek, Andrzej and Reiher, Christian  and Ruci\'nski, Andrzej and Schacht, Mathias},
  journal={Random Structures \& Algorithms},
  volume={56},
  pages={122--141},
  year={2020},
}

@article{joos2,
  title={Spanning trees in randomly perturbed graphs},
  author={Joos, Felix and Kim, Jaehoon},
  journal={Random Structures \& Algorithms},
  volume={56},
  pages={169--219},
  year={2020},
}

@article{nt,
  title={Sprinkling a few random edges doubles the power},
  author={Nenadov, Rajko and Truji\'c, Milos},
  journal={arXiv:1811.09209},
  year={2018},
}

@article{hmt,
  title={Tilings in randomly perturbed graphs: bridging the gap between {H}ajnal--{S}zemer\'edi and {J}ohansson--{K}ahn--{V}u },
  author={Han, Jie and Morris, Patrick and Treglown, Andrew},
  journal={Random Structures \& Algorithms},
  year={to appear},
}

@article{szemeredi,
  title={Regular partitions of graphs},
  author={Szemer\'edi, Endre},
  pages={399--401},
  year={1978},
  journal={Proc. {C}olloque {I}nter. {C.N.R.S.} No 260 -- {P}robl\`emes {C}ombinatoires et {T}h\'eorie des {G}raphes, Orsay},
}

@article{komsim,
  title={Szemer\'edi's {R}egularity {L}emma and its applications in graph theory},
  author={Koml\'os, J. and Simonovits, M.},
  journal={Combinatorics, {P}aul {E}rd\H{o}s is {E}ighty (Volume 2), Keszthely, Hungary, 1993, (D.~Mikl\'os, V.~T.~S\'os, T.~Sz\"onyi eds.), Bolyai Math. Stud., Budapest},
  pages={295--352},
  year={1996},
}

@article{turan,
  title={On an extremal problem in graph theory (in {H}ungarian)},
  author={Tur\'an, P.},
  journal={Math. Fiz. Lapok},
  volume={48},
  year={1941},
  pages={436--452},
}

\end{document}